\documentclass[11pt]{article}
\usepackage[french,english]{babel}
\usepackage[applemac]{inputenc}
\usepackage[T1]{fontenc}
\usepackage{enumerate}
\usepackage{verbatim}
\usepackage{amssymb,url,xspace,amsmath,amsthm,eucal}
\usepackage{mathrsfs,eufrak,array,epsfig,psfrag,graphicx,bbm,stmaryrd}
\usepackage[all]{xy}
\usepackage[a4paper,vmargin={3.5cm,3.5cm},hmargin={2.5cm,2.5cm}]{geometry}
\usepackage[font=sf, labelfont={sf,bf}, margin=1cm]{caption}
\usepackage[pdftex]{hyperref}
\usepackage{graphicx}

\newtheorem{thm}{Theorem}[section]
\newtheorem{lem}[thm]{Lemma}
\newtheorem*{lem*}{Lemma}
\newtheorem{prop}[thm]{Proposition}
\newtheorem{cor}[thm]{Corollary}

\theoremstyle{definition}

\newtheorem{defn}[thm]{Definition}

\newtheorem{conj}[thm]{Conjecture}
\newtheorem{rem}[thm]{Remark}

\newcommand{\ams}[2]{  \noindent {\footnotesize
             {\small \em AMS {\rm 2000} subject classifications.
             {\rm Primary {\sc #1}; secondary {\sc #2}} } } }

\renewcommand\Pr[1]{\mathbb{P}\left(#1\right)}

\newcommand\Prmu[1]{\mathbb{P}_{\mu}\left(#1\right)}

\newcommand\Es[1]{\mathbb{E}\left[#1\right]}

\newcommand \Prmuln[1] {\mathbb{P}_\mu\left( #1 \, | \, \lt=n\right)}

\def \L {\mathbf{D}}

\def \P {\mathbb{P}}

\def \Pmu {\mathbb{P}_\mu}

\def \N {\mathbb N}

\def \D {\mathbb D}
\def \Db {\overline{\mathbb D}}

\def \R {\mathbb R}

\def \D {\mathbb D}
\def \T {\mathbf T}

\newcommand{\keywords}[1]{ \noindent {\footnotesize
             {\small \em Keywords and phrases.} {\sc #1} } }

\def \bf {\mathfrak{f}}

\def \d {\displaystyle}
\def \a {\alpha}
\def \b {\beta}

\def \z {\zeta}

\def \t {\mathcal{T}}

\def \| { \, | \,}

\def \zt {\zeta(\tau)}
\def \lt {\lambda(\tau)}

\date{}

\author{\large Nicolas Curien\footnote{\'Ecole Normale Sup\'erieure, nicolas.curien@ens.fr} \qquad Igor Kortchemski \footnote{Universit\'e Paris-Sud, igor.kortchemski@normalesup.org} }

\long\def\symbolfootnote[#1]#2{\begingroup%
\def\thefootnote{\fnsymbol{footnote}}\footnote[#1]{#2}\endgroup}

\begin{document}
\begin{center} { \huge Random non-crossing plane configurations: \\ A conditioned Galton-Watson tree approach}\end{center}

\medskip

\begin{center}
{\large Nicolas Curien\symbolfootnote[1]{\'Ecole Normale Sup\'erieure, nicolas.curien@ens.fr} \qquad Igor Kortchemski \symbolfootnote[2]{Universit\'e Paris-Sud, igor.kortchemski@normalesup.org} }\end{center}

\bigskip

\begin{abstract}
We study various models of random non-crossing configurations consisting of diagonals of convex polygons, and focus in particular on uniform
dissections and non-crossing trees. For both these models, we
prove convergence in distribution towards Aldous' Brownian
triangulation of the disk. In the case of dissections, we also
refine the study of the maximal vertex degree and validate a
conjecture of Bernasconi, Panagiotou and Steger. Our main tool is the use of an underlying Galton-Watson tree structure.
\end{abstract}
\medskip 

\ams{60J80,60J05}{05C30,05C10,05C07}
\section*{Introduction}

Various models of non-crossing geometric configurations involving diagonals of a convex polygon in the plane have been studied both in
geometry, probability theory and especially in enumerative
combinatorics (see e.g. \cite{FN99}). Three specific models of
non-crossing configurations -- triangulations, dissections and
non-crossing trees -- have drawn particular attention. 
Let us first recall the definition of these models.

Let $P_{n}$ be the convex polygon inscribed in the unit disk of 
the complex plane whose vertices are the $n$-th roots of unity.
By definition,
a \emph{dissection} of $P_n$ is the union of the sides of $P_n$
 and of a collection of diagonals that may intersect only at
their endpoints. A \emph{triangulation} is a dissection 
whose inner faces are all triangles. Finally, a \emph{non-crossing
tree} of $P_n$ is a tree drawn on the plane whose vertices are all
vertices of $P_n$ and whose edges are non-crossing line
segments. See Fig.\,\ref{fig1} for examples.

\begin{figure}[!h]
 \begin{center}
 \includegraphics[height=3cm]{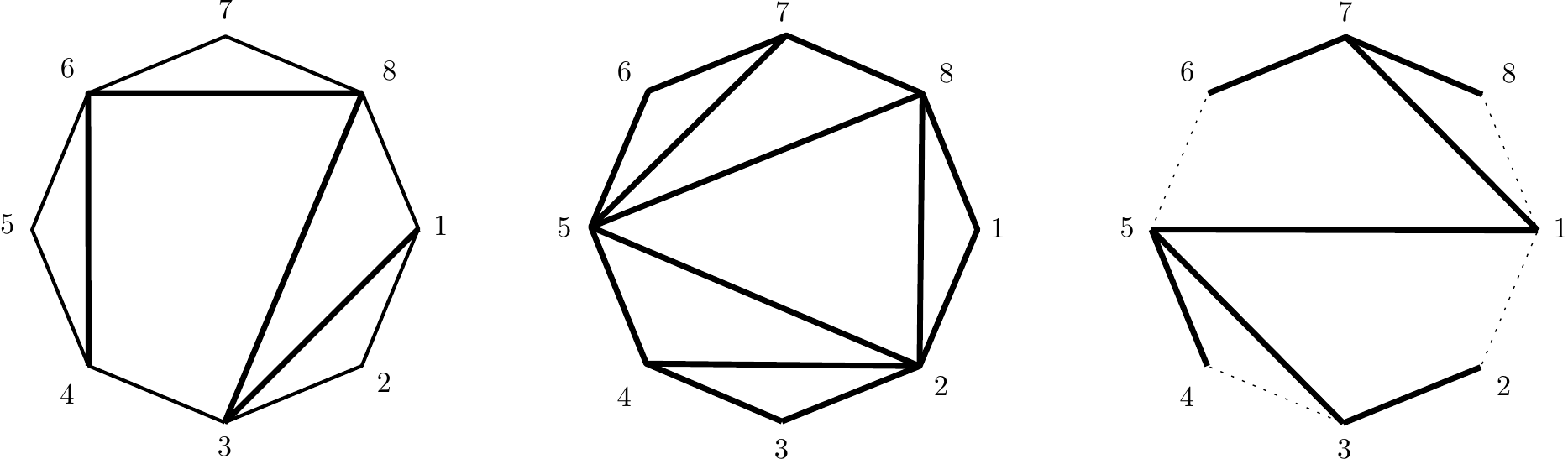}
 \caption{ \label{fig1} A dissection, a triangulation and a non-crossing tree of the octogon.}
 \end{center}
 \end{figure}


Graph theoretical properties of uniformly distributed triangulations have
been recently investigated in combinatorics. For instance, the study of the
asymptotic behavior of the maximal vertex degree 
has been initiated in \cite{DFHN99} and pursued in \cite{GW00}.
Afterwards, the same random variable has been studied in the case of
dissections \cite{BPS08}.

 We shall continue the study of
graph-theoretical properties of large uniform dissections and in
particular focus on the maximal vertex degree. Our method is based
on finding and exploiting an underlying Galton-Watson tree
structure. More precisely, we show that the \emph{dual tree}
associated with a uniformly distributed dissection of $P_n$ is a
critical Galton-Watson tree conditioned on having exactly $n-1$
\emph{leaves}. This new conditioning of Galton-Watson trees has been studied recently in \cite{K:leaves} (see also \cite{Rizzolo})
and is well adapted to the study of dissections, see
\cite{K:lam}.  In particular we are able to
validate a conjecture contained in \cite{BPS08} concerning the
asymptotic behavior of the maximal vertex degree in a uniform
dissection (Theorem \ref{thm:vertexdegrees}). Using the critical
Galton-Watson tree conditioned to survive introduced in
\cite{Kes86}, we also give a simple probabilistic explanation of the
fact that the inner degree of a given vertex in a large uniform
dissection converges in distribution to the sum of two independent
geometric variables (Proposition \ref{prop:somme2geom}). We finally
obtain new results about the asymptotic behavior of the maximal
face degree in a uniformly distributed dissection. 

As a by-product of
our techniques, we give a very simple probabilistic approach to the following enumeration problem. Let $\mathcal{A}$ be a non-empty subset of $\{3,4,5,\ldots\}$ and $\mathbf{D}^{(\mathcal{A})}_n$ the set of all dissections of $P_{n+1}$ whose face degrees all belong to the set $\mathcal{A}$. Theorem \ref{prop:gen} gives an explicit asymptotic formula for $\# \mathbf{D}_n^{( \mathcal{A})}$ as $n \rightarrow \infty$ (for those values of $n$ for which $ \mathbf{D}_{n}^{( \mathcal{A})} \ne  \varnothing$).  In particular when $\mathcal{A}_0=\{3,4,5,\ldots\}$,  then $ \mathbf{D}_{n-1}:=\mathbf{D}^{(\mathcal{A}_0)}_{n-1}$ is the set of all dissections of $P_n$ and
$$  \#\mathbf{D}_{n-1} \quad \mathop{\sim}_{n
\to \infty} \quad \frac{1}{4}\sqrt{\frac{99 \sqrt{2}-140}{\pi}}
n^{-3/2} (3+2\sqrt{2})^n.$$ This formula (Corollary \ref{prop:comptage}) was originally derived by Flajolet \& Noy \cite{FN99} using very different techniques.

\bigskip

From a geometrical perspective, Aldous \cite{Ald94a,Ald94b} proposed to consider
triangulations of $P_n$ as
closed subsets of the unit disk $ \overline{ \mathbb{D}} := \{ z \in \mathbb{C} : |z| \leq 1\}$ rather than viewing them as graphs. He proved that large uniform
triangulations of $P_n$ converge in distribution (for the Hausdorff
distance on compact subsets of the unit disk) towards a random
compact subset. This limiting object is called \emph{the Brownian
triangulation} (see Fig.\,\ref{brownian}). This
name comes from the fact that the Brownian triangulation can be constructed from the
Brownian excursion as follows: Let $\mathbbm{e} : [0,1] \to
\mathbb{R}$ be a normalized excursion of linear Brownian motion. For
every $s,t \in [0,1]$, we set $s \sim t$ if we have $ \mathbbm{e}(s)
= \mathbbm{e}(t) = \min_{[s\wedge t, s \vee t]} \mathbbm{e}$. The
Brownian triangulation is then defined as:
 \begin{eqnarray} \label{defbl}\mathcal{B} &:=& \bigcup_{s \sim t} \big[e^{-2\textrm{i} \pi s},e^{-2\textrm{i} \pi t}\big],\end{eqnarray}
 where $[x,y]$ stands for the Euclidean line segment joining two complex
 numbers $x$ and $y$.
 \begin{figure}[h!]
 \begin{center}
 \includegraphics[height=4cm]{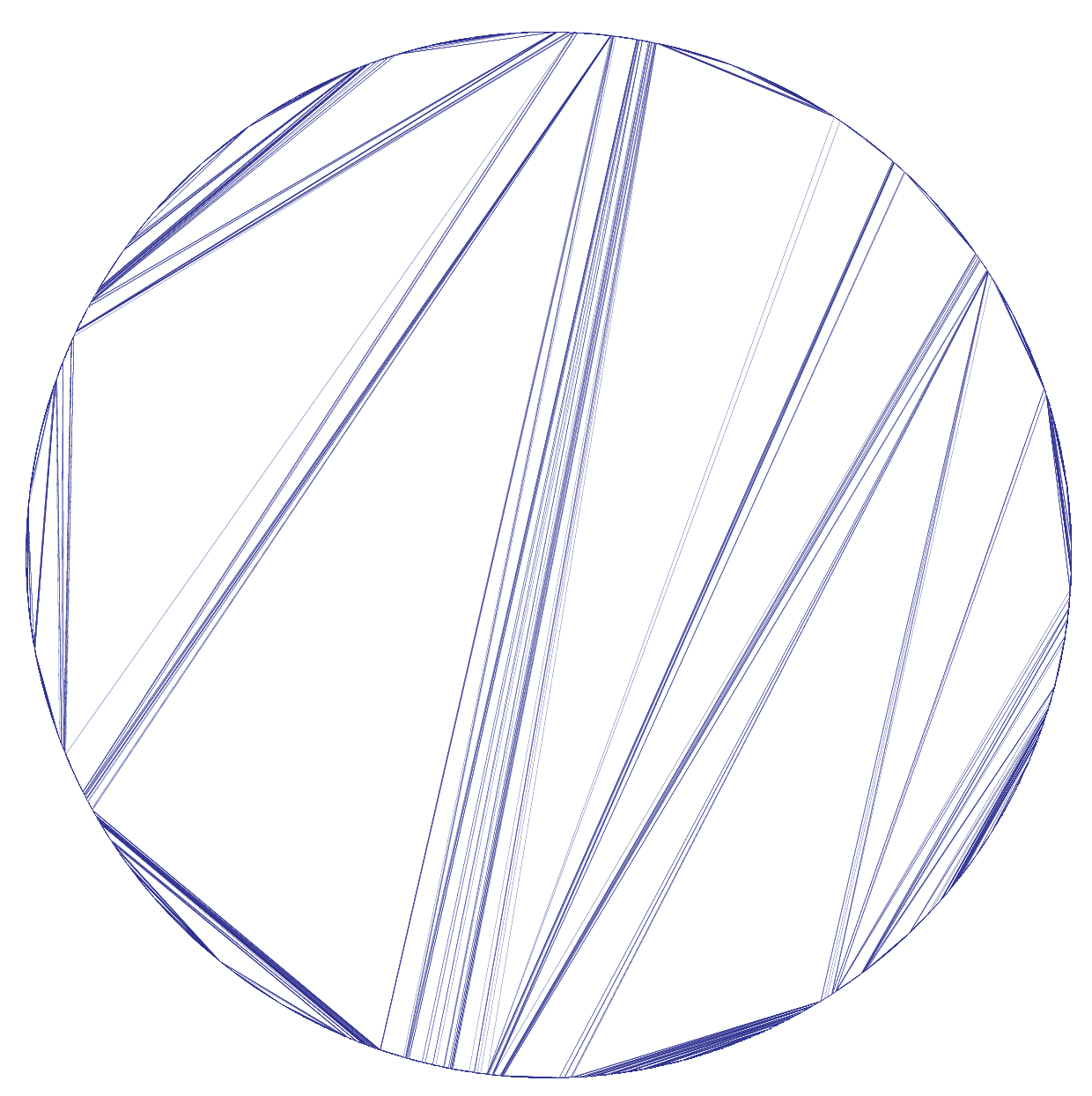}
 \caption{\label{brownian} A sample of the Brownian triangulation $ \mathcal{B}$.}
 \end{center}
 \end{figure}
It turns out that $\mathcal{B}$ is almost surely a closed set which is a
continuous triangulation of the unit disk (in the sense that the
complement of $ \mathcal{B}$ in $\overline{ \mathbb{D}}$ is a
disjoint union of open Euclidean triangles whose vertices belong to
the unit circle). Aldous also observed that the Hausdorff dimension of $ \mathcal {B}$ is almost surely equal to $3/2$ (see also \cite{LGP08}). Later, in the context of random maps, the Brownian
triangulation has been studied by Le Gall \& Paulin in \cite{LGP08}
where it serves as a tool in the proof of the homeomorphism theorem
for the Brownian map. See also \cite{CLG10,K:lam} for analogs of the Brownian triangulation.

However, neither large random 
{uniform}
dissections, nor large uniform non-crossing trees
have yet been studied from this geometrical point of view.
\begin{figure}[!h]  \begin{center}  \includegraphics[height=3cm]{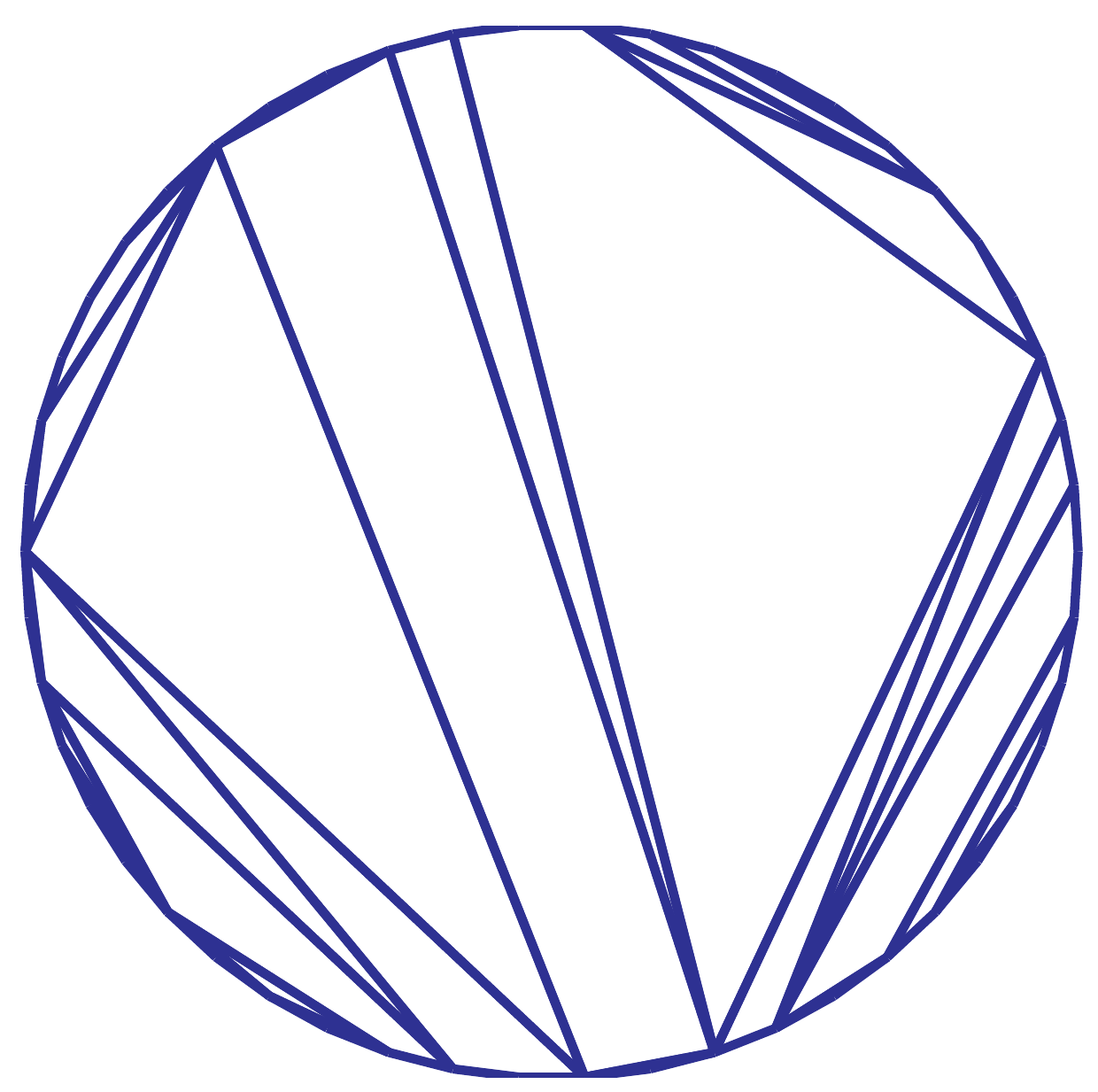} \hspace{0.5cm}
\includegraphics[height=3cm]{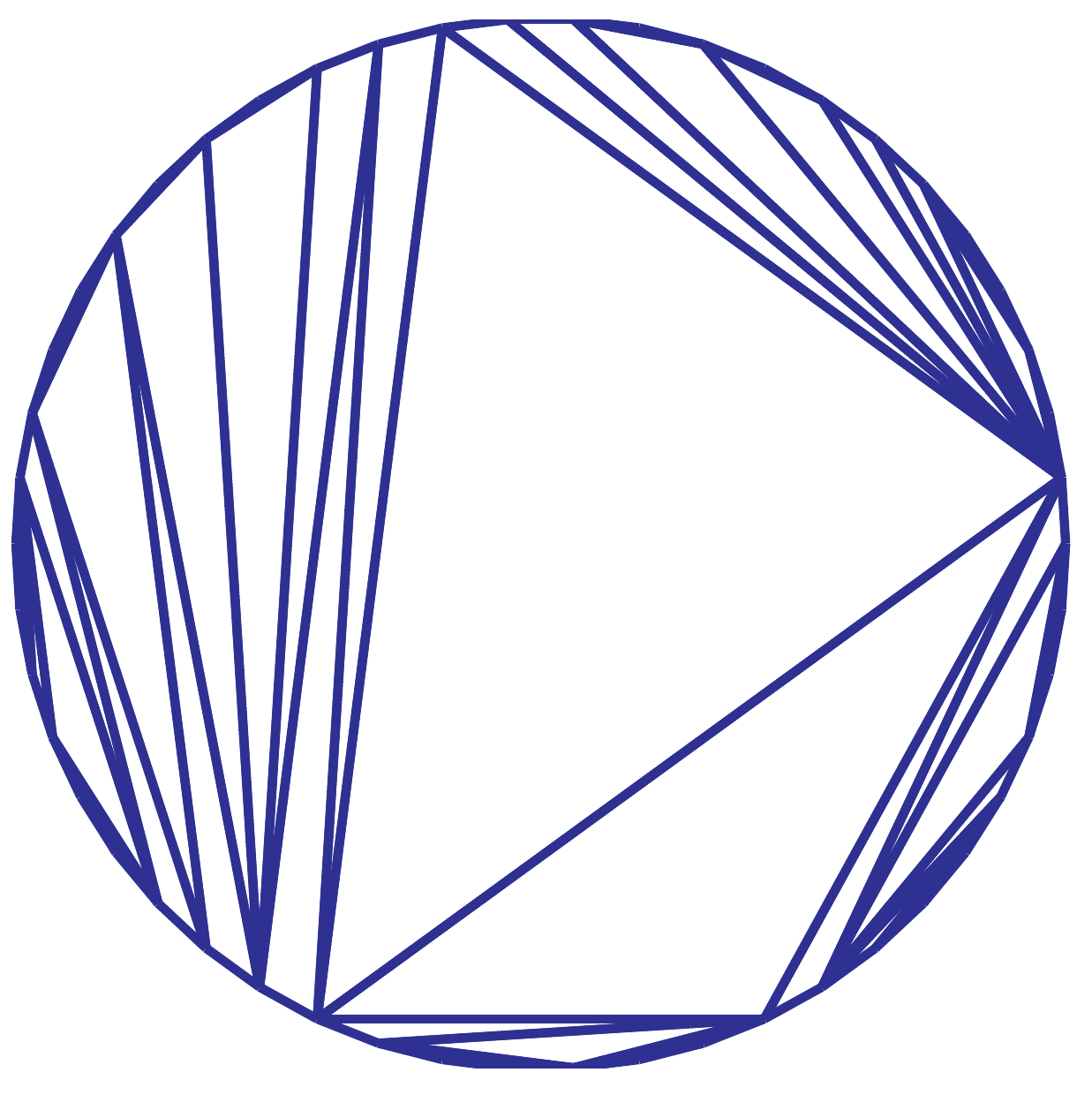}
\hspace{0.5cm}
\includegraphics[height=3cm]{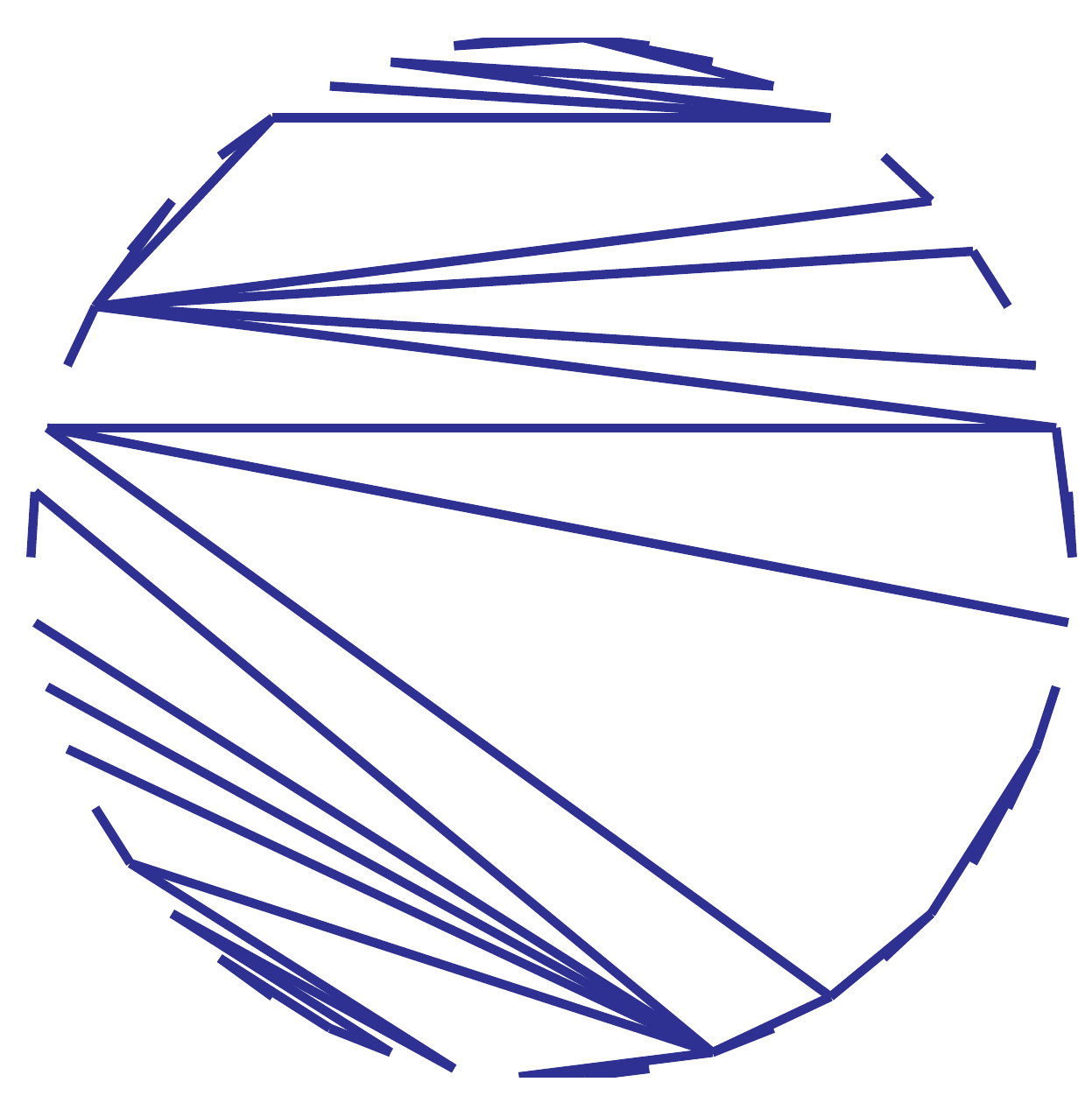}    \caption{
\label{fig2} Uniform dissection, triangulation and non-crossing tree
of size $50$. The same continuous model?}  \end{center}
\end{figure} 

\noindent In this work, we extend Aldous' theorem by showing that
both large uniform dissections and large uniform non-crossing trees
converge in distribution towards the Brownian triangulation (Theorem
\ref{thm:brownianlimit}). The maybe surprising fact that large
uniform dissections (which may have non-triangular faces) converge
to a continuous triangulation stems from the fact that many
diagonals degenerate in the limit.
For both models, the key is to use a Galton-Watson tree structure, which was
already described above in the case of dissections. In the case of
non-crossing trees this structure has been identified by Marckert
\& Panholzer \cite{MP02} who established that the \emph{shape} of a
uniform non-crossing tree of $P_n$ is \emph{almost} a Galton-Watson
tree conditioned on having $n$ vertices (see Theorem
\ref{thm:MarckertPanholzer} below for a precise statement).

We also consider other random configurations of non-crossing diagonals of $P_n$ such as  non-crossing graphs, non-crossing partitions and non-crossing pair partitions, and prove the convergence towards the Brownian triangulation, again by using an appropriate underlying tree structure. We also show that a uniformly distributed dissection over  $\mathbf{D}^{(\mathcal{A})}_{n}$ converges towards $ \mathcal {B}$ as $n \rightarrow \infty$. The Brownian triangulation thus appears as a universal limit for random non-crossing configurations. This has interesting applications: For instance, let  $ \chi_n$ be a random non-crossing configuration on the vertices of $P_n$  that converges in distribution towards $ \mathcal{B}$ in the sense of the Hausdorff metric. 
Then the Euclidian length of the longest diagonal of $ \chi_n$ converges in distribution towards the length of the longest chord of the Brownian triangulation with law $$\frac{1}{\pi} \frac{3x-1}{x^2(1-x)^2 \sqrt{1-2x}} \mathbf{1}_{\frac{1}{3} \leq x \leq \frac{1}{2}} \mathrm{d}x.$$
This has been shown in the particular case of triangulations in \cite{Ald94b}  (see also \cite{DFHN99}). 

\medskip

The paper is organized as follows. In the first section we introduce
the discrete models and explain the underlying Galton-Watson tree
structures. The second section is devoted to the convergence of
different random non-crossing configurations towards the
Brownian triangulation and to applications. The final section
contains the analysis of some graph-theoretical properties of large
uniform dissections, such as the maximal vertex or face degree.
\bigskip

\noindent \textbf{Acknowledgments.} We are indebted to Jean-Fran\c
cois Le Gall for a very careful reading of a first version of this
article and for useful suggestions. We thank the participants of the
ANR A3 meeting in Nancy (March 2011) for stimulating discussions and
especially Jean-Fran\c cois Marckert for telling us about
non-crossing trees.

\section{Dissections, non-crossing trees and  Galton-Watson trees}

\subsection{Dissections and plane trees}

Throughout this work, for every integer $n\geq3$, $P_{n}$ stands for
the regular polygon of the plane with $n$ sides whose vertices are
the $n$-th roots of unity.

\begin{defn}A \emph{dissection} $\mathcal{D}$ of the polygon $P_{n}$ is  the union of the sides of the polygon and of a collection of diagonals that may intersect only at
their endpoints. A \emph{face} $f$ of $ \mathcal{D}$ is a connected component of the complement of $ \mathcal{D}$
inside $P_n$; its degree, denoted by $\deg(f)$, is the number of sides
surrounding $f$. See Fig. \ref{fig:dual} for an example.\end{defn}

 Let $\L_n$ be the set of all dissections of $P_{n+1}$. Given a dissection $ \mathcal{D} \in \L_n$, we construct a (rooted ordered) tree $\phi( \mathcal{D})$ as follows: Consider the dual graph
of $ \mathcal{D}$, obtained by placing a vertex inside each face of
$ \mathcal{D}$ and outside each side of the polygon $P_{n+1}$ and by
joining two vertices if the corresponding faces share a common edge,
thus giving a connected graph without cycles. Then remove the dual
edge intersecting the side of $P_{n+1}$ which connects $1$ to
$e^{\frac{2 \textrm{i} \pi}{n+1}}$. Finally, root the tree at the corner
adjacent to the latter side (see Fig. \ref{fig:dual}).

\begin{figure*}[!h]
\begin{center}
\includegraphics[scale=0.6]{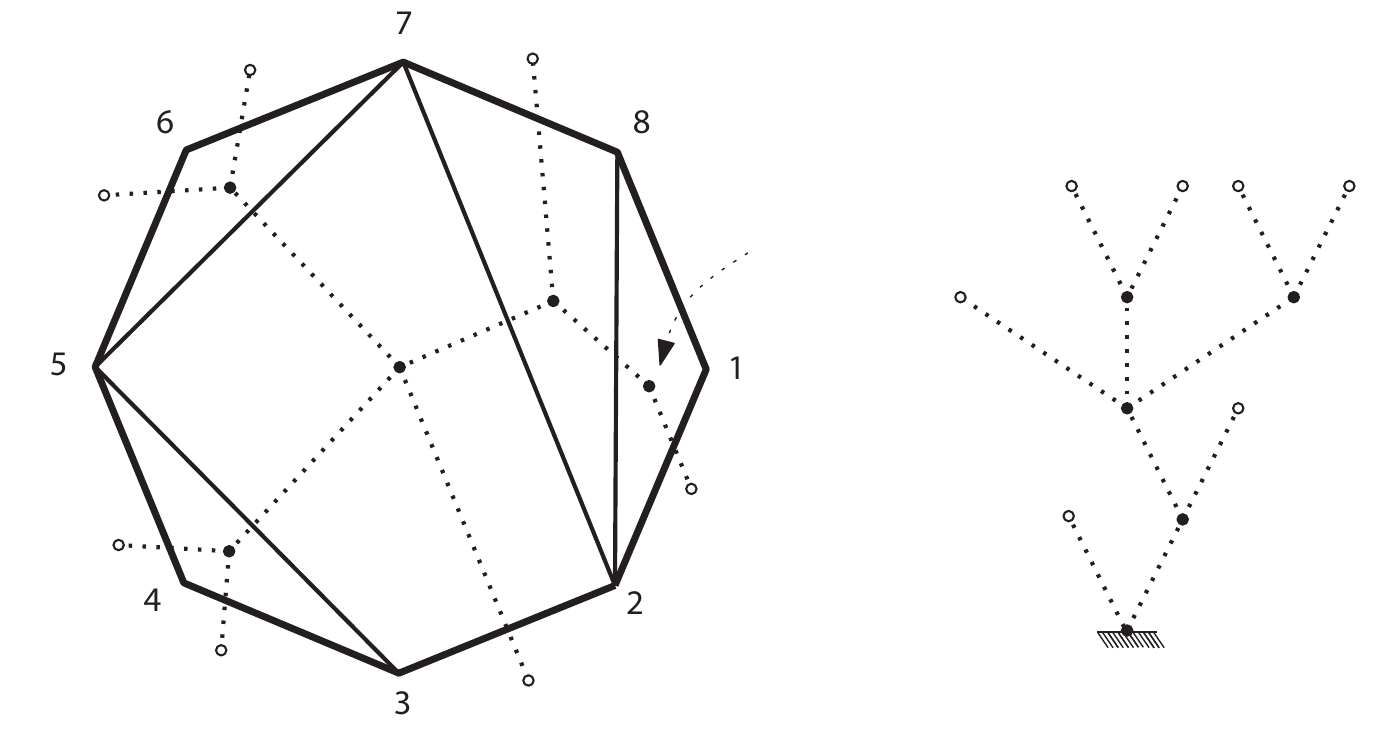}
\caption{\label{fig:dual} The dual tree of a dissection of $P_{8}$, note that the tree has $7$ leaves.}
\end{center}
\end{figure*}

The dual tree of a dissection is a plane  tree (also known as
rooted ordered tree in the literature). We briefly recall the
formalism of plane trees which can be found in \cite{LG06} for
example. Let $\N=\{0,1,\ldots\}$ be the set of nonnegative integers,
$\N^*=\{1,\ldots\}$ and let $ \mathcal{U}$ be the set of labels
$$ \mathcal{U}=\bigcup_{n=0}^{\infty} (\N^*)^n,$$
where by convention $(\N^*)^0=\{\varnothing\}$. An element of $ \mathcal{U}$ is
a sequence $u=u_1 \cdots u_m$ of positive integers, and we set
$|u|=m$, which represents the \og generation \fg \  of $u$. If $u=u_1
\cdots u_m$ and $v=v_1 \cdots v_n$ belong to $ \mathcal{U}$, we write $uv=u_1
\cdots u_m v_1 \cdots v_m$ for the concatenation of $u$ and $v$. 
Finally, a \emph{plane tree} $\tau$ is a finite or infinite subset of $
\mathcal{U}$ such that:
\begin{itemize}
\item[1.] $\varnothing \in \tau$,
\item[2.] if $v \in \tau$ and $v=uj$ for some $j \in \N^*$, then $u
\in \tau$,
\item[3.] for every $u \in \tau$, there exists an integer $k_u(\tau)
\geq 0$ (the number of children of $u$) such that, for every $j \in \N^*$, $uj \in \tau$ if and only
if $1 \leq j \leq k_u(\tau)$.
\end{itemize}
In the following, by \emph{tree} we will always mean plane tree. We denote the set of all trees by $\T$. 
We will often view each vertex of a tree $\tau$ as an individual of
a population whose $\tau$ is the genealogical tree. If $\tau$ is a tree and $u \in \tau$, we define the shift of $\tau$ at $u$ by $\sigma_u \tau=\{v \in U : \, uv \in \tau\}$, which is itself a tree. 
 If $u, v \in \tau$ we denote by $\llbracket u,v \rrbracket$ the discrete geodesic path between $u$ and $b$ in $\tau$.
The total
progeny of $\tau$, which is the total number of vertices of $\tau$,
will be
denoted by $\zeta(\tau)$. The number of leaves (vertices $u$ of $ \tau$ such that $k_u( \tau)=0$) of the tree $\tau$ is denoted by $\lambda( \tau)$.  Finally, we let $ \T^{(\ell)}_n$ denote the set of all plane trees with $n$ leaves such that there is no vertex with exactly one child.

\medskip

The following proposition is an easy combinatorial property, whose
proof is omitted.

\begin{prop}\label{prop:bijection} The duality application $\phi$ is a bijection between $ \mathbf{D}_{n}$ and $ \T^{(\ell)}_n$. \end{prop}

Finally, we briefly recall the standard definition of Galton-Watson
trees.

\begin{defn}Let $\rho$ be a probability measure on $\N$ such that $\rho(1)<1$.  
The law of the
Galton-Watson tree with offspring distribution $\rho$ is the unique
probability measure $\P_\rho$ on $\T$ such that:
\begin{itemize}
\item[1.] $\P_\rho(k_\varnothing=j)=\rho(j)$ for $j \geq 0$,
\item[2.] for every $j \geq 1$ with $\rho(j)>0$, conditionally on $\{ k_{\varnothing}=j\}$, the subtrees
$\sigma_1 \tau, \ldots, \sigma_j \tau$ are i.i.d.\,with distribution $\P_\rho$.
\end{itemize}
It is well known that if $\rho$ has mean less than or equal to $1$ then a $\rho$-Galton-Watson tree is almost surely finite.  In the sequel, for every integer $j \geq 1$, $ \mathbb{P}_{\rho,j}$
will stand for the probability measure on $\T^j$ which is the
distribution of $j$ independent trees of law $ \mathbb{P}_{\rho}$.
The canonical element of $\T^j$ will be denoted by $ \mathfrak{f}$.
For $\bf=(\tau_1,\ldots,\tau_j) \in \T^j$, let $\lambda(\bf) =
\lambda(\tau_1)+\cdots+\lambda(\tau_j)$ be the total number of
leaves of $\bf$.
\end{defn}


\subsection{Uniform dissections are conditioned Galton-Watson trees}

In the rest of this work, $ \mathcal{D}_{n}$ is a random dissection uniformly distributed over $ \mathbf{D}_{n}$. We also set $\t_{n} = \phi( \mathcal{D}_{n})$ to simplify notation. Remark that $\t_{n}$ is a random tree which belongs to $ \T^{(\ell)}_n$.

Fix $c \in (0,1/2)$ and define a probability distribution $\mu^{(c)}$ on $\N$ as
follows:
$$\mu^{(c)}_0 = \frac{1-2c}{1-c}, \qquad \mu^{(c)}_1 = 0, \qquad \mu^{(c)}_{i} = c^{i-1} \mbox{ for } i \geq 2.$$
It is straightforward to check that $\mu^{(c)}$ is a probability
measure, which moreover has mean equal to $1$ when $c
=1-2^{-1/2}$. In the latter case, we drop the exponent ${(c)}$
in the notation, so that $ \mu:= \mu^ {(1- 1/\sqrt {2})}$. The following theorem gives a
connection between uniform dissections of $P_n$ and Galton-Watson
trees conditioned on their number of leaves. This connection has been obtained independently of the present work in \cite{PitmanRizzolo}.


\begin{prop}\label{prop:unif}The conditional probability distribution $ \mathbb{P}_{\mu^{(c)}}( \, \cdot \mid \lambda(\tau) =
n)$ does not depend on the choice of $c \in (0,1/2)$ and coincides with the distribution of the dual tree $\t_{n}$ of a uniformly distributed dissection of $P_{n+1}$.\end{prop}

\begin{proof}We adapt the proof of \cite[Proposition 1.8]{K:lam} in our context. 
By Proposition \ref{prop:bijection}, it is sufficient to show that
for every $c \in (0,1/2)$ the probability distribution  $
\mathbb{P}_{\mu^{(c)}}( \cdot \mid \lambda(\tau) = n)$ is the
uniform probability distribution over $\T^{(\ell)}_n$. If $ \tau$ is a tree, we denote by $u_0,\ldots,u_{\zeta( \tau)-1}$ the vertices
of $ \tau$ listed in lexicographical order and recall that $k_{u_i}$
stands for the number of children of $u_i$. Let $ \tau_0 \in \T^{(\ell)}_n$. By the definition of $ \mathbb{P}_{\mu^{(c)}}$,
we have
 \begin{eqnarray*} \mathbb{P}_{\mu^{(c)}}( \tau=\tau_0 \mid \lambda(\tau) = n)&=&\frac{1}{\mathbb{P}_{\mu^{(c)}}(\lt=n)}  \prod_{i=0}^{\zeta(\tau_0)-1}
\mu^{(c)}_ {k_{u_i}}. \end{eqnarray*} Using the definition of $\mu^{(c)}$, the
product appearing in the last expression can be written as
$$ \prod_{i=0}^{\zeta(\tau_0)-1}
\mu^{(c)}_ {k_{u_i}}= \left( \frac{1-2c}{1-c}\right)^{  \lambda( \tau_0)} c^{  \z(\tau_0)- 1- (\zeta(\tau_0)-\lambda(\tau_0))}
=c^{-1}\left(\frac{c(1-2c)}{1-c}\right)^{  \lambda(\tau_0)}.$$Thus
$\mathbb{P}_{\mu^{(c)}}( \tau=\tau_0 \mid \lambda(\tau) = n)$ depends
only on $\lambda(\tau_0)$. We conclude that $\mathbb{P}_{\mu^{(c)}}( \,
\cdot \mid \lambda(\tau) = n)$ is the uniform distribution over $ \T^{(\ell)}_n$.\end{proof}

 In the following,
 we will always choose $c = 1-2^{-1/2}$ for $\mu^{(c)}=\mu$ to be critical. Hence, the random tree $\t_{n}$ has law $\mathbb{P}_{\mu}( \cdot \mid
\lambda(\tau) = n)$.   A general study of Galton-Watson trees conditioned by their number of leaves is made in \cite{K:leaves}. In particular, we will make an
extensive use of the following asymptotic estimate which is a particular case of  \cite[Theorem 3.1]{K:leaves}:
 \begin{lem} \label{estimee} We have  \begin{eqnarray}\label{eq:estimee}\Prmu{\lt = n}
&\d \mathop{\sim}_{n \to \infty} &
\frac{n^{-3/2}}{2\sqrt{\pi\sqrt{2}}}.
 \end{eqnarray}\end{lem}


Let us  give an application of Proposition \ref{prop:unif} and Lemma \ref{estimee} to the enumeration of dissections. There exists no easy closed formula for the
number $\# \mathbf{D}_{n}$ of dissections of $P_{n+1}$. However, a recursive
decomposition easily shows that the generating function
 \begin{eqnarray*} D(z) &:=& \sum_{n \geq 3} z^n \# \mathbf{D}_{n-1}, \end{eqnarray*}
is equal to $\frac{z}{4} (1+z - \sqrt{z^2-6z+1})$, see e.g.
\cite[Section 3]{BPS08} and \cite{FN99}. Using classical techniques
of analytic combinatorics \cite{FS09}, it is then possible to get the
asymptotic behavior of $ \# \mathbf{D}_{n}$, see \cite{FN99}.
Here, we present a very short ``probabilistic'' proof of this result.

\begin{cor}[Flajolet \& Noy, \cite{FN99}]\label{prop:comptage}We have
 $$\#\L_{n-1}  \quad \mathop{\sim}_{n \to \infty} \quad \frac{1}{4}\sqrt{\frac{99 \sqrt{2}-140}{\pi}} n^{-3/2} (3+2\sqrt{2})^n.$$
\end{cor}

\begin{proof} 
Let $n \geq 3$ and let $ \tau_0=\{\varnothing,1,2,\ldots,n-1\}$ be the tree consisting of the root and its $n-1$ children.
By Proposition \ref{prop:unif}, we have
$$\frac{1}{\#\L_{n-1}}=\Pmu(\tau=\tau_0 \, | \, \lt=n-1)=\frac{\Pmu(\tau=\tau_0)}{\Pmu(\lt=n-1)}=
\frac{\mu_{n-1} \mu_0^{n-1}}{\Pmu(\lt=n-1)}.$$ Thus
 \begin{eqnarray} \#\L_{n-1}=\frac{\Pmu(\lt=n-1)}{(2-\sqrt{2})^{n-1}
\left(\frac{2-\sqrt{2}}{2}\right)^{n-2}}=\frac{(2-\sqrt{2})^3}{4}
\frac{\Pmu(\lt=n-1)}{(3-2 \sqrt{2})^n}. \label{equaun}\end{eqnarray}
The statement of the corollary now follows from  \eqref{eq:estimee} and \eqref{equaun}.\end{proof}

\subsection{Non-crossing trees are \emph{almost} conditioned Galton-Watson trees}

\begin{defn} A \emph{non-crossing tree} $ \mathscr{C}$ of $P_{n}$ is a tree
drawn in the plane whose vertices are all the vertices of $P_{n}$ and
whose edges are Euclidean line segments that do not intersect except possibly at
their endpoints.
\end{defn}

Every non-crossing tree $ \mathscr{C}$ inherits a plane
tree structure by rooting $\mathscr{C}$ at the vertex $1$ of $P_{n}$ and keeping the planar ordering
induced on $\mathscr{C}$. The children of the root vertex are ordered by going in clockwise order around the point $1$ of $P_n$, starting from the edge connecting $1$ to $ e^ {- 2 \textrm {i} \pi/n}$, which may or may not be in $ \mathcal {C}$. As in \cite{MP02}, we call this plane tree
\emph{the shape} of $\mathscr{C}$ and denote it by $ S ( \mathscr{C})$. Obviously $ \z( S(\mathscr{C}))=n$. Note that the mapping $ \mathscr{C} \mapsto S( \mathscr{C})$ is not one-to-one. 
However, we will later see that large scale properties of  uniform non-crossing trees are governed by their shapes. \medskip

In the following, we let $\mathscr{C}_{n}$ be uniformly distributed over the set of all non-crossing
trees of $P_{n}$. We also set $ \mathscr{T}_{n} = S(
\mathscr{C}_{n})$ to simplify notation. We start by
recalling a result of Marckert and Panholzer stating that $
\mathscr{T}_{n}$ is \emph{almost}  a Galton-Watson tree. Consider
the two offspring distributions:
$$ \begin{array}{rclc}
\nu_{\varnothing}(k) &=& 2 \cdot 3^{-k}, &\mbox{for }k=1,2,3, ...\\
\nu(k) &=& \displaystyle 4(k+1)3^{-k-2}, & \mbox{for }k=0,1,2,...
\end{array}$$
Following \cite{MP02}, we introduce a \emph{modified} version of the
$\nu$-Galton-Watson tree where the root vertex has a number of children
distributed according to $\nu_{ \varnothing}$ and all  other individuals have
offspring distribution $\nu$. We denote the resulting probability measure on plane trees
obtained by $ \widetilde{ \mathbb{P}}_{\nu}$. The following theorem is the main
result of \cite{MP02} and will be useful for our purposes:

\begin{thm}[Marckert \& Panholzer, \cite{MP02}] \label{thm:MarckertPanholzer} The random plane tree
$ \mathscr{T}_{n}$ is distributed according to
$\widetilde{\mathbb{P}}_{\nu}( \cdot \mid \zeta(\tau) = n)$.
\end{thm}

\section{The Brownian triangulation: A universal limit for random non-crossing configurations}

Recall that $ \mathcal{D}_{n}$ is a uniform dissection of $P_{n+1}$
and that $ \mathcal{T}_{n}$ stands for its dual plane tree. Recall
also that $ \mathscr{C}_{n}$ is a uniform non-crossing tree of $P_n$
and that $ \mathscr{T}_{n}$ stands for its shape. In the following,
we will view both $ \mathcal{D}_{n}$ and $ \mathscr{C}_{n}$ as
random closed subsets of $ \overline{ \mathbb{D}}$ as suggested by Fig.\,\ref {fig1}.
 Recall that the \emph{Hausdorff distance}
between two closed subsets of $A,B \subset \overline{ \mathbb{D}}$
is
 \begin{eqnarray*} \mathrm{d_{Haus}}(A,B) &=& \inf\big\{ \varepsilon >0 : A \subset B^{(\varepsilon)}
 \mbox{ and }B \subset A^{(\varepsilon)}\big\}, \end{eqnarray*}
 where $X^{(\varepsilon)}$ is the $ \varepsilon$-enlargement of a set $X \subset     \overline{ \mathbb{D}}$.
 The set of all closed subsets of $ \overline{ \mathbb{D}}$ endowed with the Hausdorff distance is a compact
 metric space. Recall that the Brownian triangulation $\mathcal{B}$ is defined by \eqref{defbl}. The main result of this section is:

\begin{thm}\label{thm:brownianlimit} The following two convergences in distribution hold for the
Hausdorff metric on closed subsets of $ \overline{ \mathbb{D}}$:

$$  (i) \quad   \mathcal{D}_{n}  \xrightarrow[n\to\infty]{(d)}  \mathcal{B},  \hspace{4cm} (ii) \quad   \mathscr{C}_{n}  \xrightarrow[n\to\infty]{(d)}  \mathcal{B}.$$

 \end{thm}

The main ingredient in the proof of Theorem \ref{thm:brownianlimit}
is a scaling limit theorem for functions coding the trees
$ \mathcal{T}_{n}$ and $ \mathscr{T}_{n}$.  In order to
state this result, let us introduce the contour function
associated to a plane tree.

 Fix a tree $\tau$ and consider a particle that starts from the
root and visits continuously all the edges of $\tau$ at unit speed
(assuming that every edge has unit length). When leaving a vertex,
the particle moves towards the first non visited child of this
vertex if there is such a child, or returns to the parent of this
vertex. Since all the edges will be crossed twice, the total time needed
to explore the tree is $2 (\zeta(\tau)-1)$. For $0 \leq t \leq
2(\zt-1)$, $C_\tau(t)$ is defined as the distance to the root of the
position of the particle at time $t$. For technical reasons, we set
$C_\tau(t)=0$ for $t \in [2(\zt-1), 2 \zt]$. The function
$C_{\tau}(\cdot)$ is called the contour function of the tree $\tau$.
See \cite{LG06} for a rigorous definition. For $ t \in [0,2( \zt-1)]$ and $ u \in \tau$, we say that the contour process visits the vertex $u$ at time $t$ if the particle is at $u$ at time $t$. Similarly, if we say that the contour process visits an edge $ \mathfrak {e}$ if the particle belongs to $ \mathfrak {e}$ at time $t$.

\bigskip

Let  $ \mathbbm{e}$ bet the normalized excursion of linear Brownian
motion. The following convergences in distribution will be useful for our purposes: \begin{eqnarray}    \label{cvcontour2} \left( \frac{C_{\mathcal{T}_{n}}(2 \z( \mathcal {T}_n)t)}{\sqrt{\z( \mathcal {T}_n)}}\right)_{0 \leq t \leq 1}
 &\xrightarrow[n\to\infty]{(d)}& \left( (3\sqrt{2}-4)^{-1/2} \mathbbm{e}(t) \right)_{0 \leq t \leq 1},\\
 \label{cvcontour}\left( \frac{C_{\mathscr{T}_{n}}(2nt)}{\sqrt{n}}\right)_{0 \leq t \leq 1}
 &\xrightarrow[n\to\infty]{(d)}& \left( 2\sqrt{\frac{2}{3}} \mathbbm{e}(t)\right)_{0 \leq t \leq
 1}.
\end{eqnarray}
The convergence \eqref{cvcontour2} has been proved by Kortchemski \cite[Theorem 5.9]{K:leaves}, and \eqref{cvcontour} has been obtained by Marckert and Panholzer \cite[Proposition 4]{MP02}.

\bigskip
The proof of Theorem
\ref{thm:brownianlimit} will be different for dissections
and non-crossing trees, although the main ideas are the same in both cases. Notice for example that in the case of non-crossing trees, there is no need to consider a dual structure since the shape of the non-crossing tree
already yields a plane tree.

\subsection{Large uniform dissections}

\subsubsection{The Brownian triangulation is the limit of large uniform dissections}

In \cite{K:lam}, a general convergence result is proved for dissections whose
dual tree is a conditioned Galton-Watson tree whose offspring distribution belongs to the domain of attraction of a
stable law. Since our approach to  the convergence of
large uniform non-crossing trees towards the Brownian triangulation
will be similar in spirit, we reproduce the main steps of the proof
in our particular finite variance case.

The following lemma, which is an easy consequence of \cite[Corollary
3.3]{K:leaves}, roughly says that leaves are distributed uniformly
in a conditioned Galton-Watson tree. Formally, if $\tau$ is a plane tree, for $0 \leq t \leq  2 \z(\tau)-2$, we let $\Lambda_{\tau}(t)$ be the number of leaves among the vertices of
$\tau$ visited by the contour process up to time $t$, and we set $\Lambda_{\tau}(t)= \lt$ for $2 \z(\tau)-2 \leq t \leq 2\zt$.

\begin{lem}\label{lem:repartitionleaves}We have
 \begin{eqnarray} \label{eq:equirep}\sup_{0 \leq t \leq 1} \left| \frac{\Lambda_{\t_n}\ \left( 2 \z(\t_n) t \right)}{n} - t \right| & \xrightarrow[n\to\infty]{(\P)} & 0, \end{eqnarray}where $(\mathbb{P})$ stands for the convergence in probability. 
\end{lem}

\proof[Proof of Theorem \ref{thm:brownianlimit} part (i)] 
We can apply Skorokhod's representation theorem (see e.g.
\cite[Theorem 6.7]{Bill}) and assume, without loss of generality,
that the convergences \eqref{cvcontour2}  and \eqref{eq:equirep} hold almost surely and we
aim at showing that $ \mathcal{D}_{n}$ converges almost surely
towards the Brownian triangulation $\mathcal{B}$ defined by
\eqref{defbl}. Since the space of compact subsets of $\Db$ equipped
with the Hausdorff metric is compact, it is sufficient to show that
the sequence $(\mathcal{D}_{n})_{n \geq 1}$ has a unique accumulation point which is $\mathcal{B}$. We fix $ \omega$ such that both convergences \eqref{cvcontour2}  and \eqref{eq:equirep} hold for this value of $ \omega$. Up to extraction, we thus suppose that
$(\mathcal{D}_{n})_{n \geq 1}$ converges towards a certain compact subset $\mathcal{D}_\infty$
of $\Db$ and we aim at showing that
$\mathcal{D}_\infty=\mathcal{B}$.

We first show that $\mathcal{B} \subset \mathcal{D}_\infty$. Fix
$0<s<t<1$ such that $ \mathbbm{e}(s)= \mathbbm{e}(t) =
\min_{[s\wedge t, s \vee t]} \mathbbm{e}$. We first consider the case when we have also  $\mathbbm{e}(r) > \mathbbm{e}(s)$ for every  $r \in (s,t)$. Let us prove that $ [e^{-2\textrm{i}\pi
s},e^{-2\textrm{i}\pi t}] \subset  \mathcal{D}_\infty$. Using the convergence
\eqref{cvcontour2}, one can find an edge of $\t_n$ such that if $s_n$ is the time of the first visit of this edge by the contour process and if $t_n$ is the time of its last visit, then $s_n / (2 \z(\mathcal {T}_n)) \rightarrow s$ and $ t_n  / (2 \z(\mathcal {T}_n)) \rightarrow t$ as $n \rightarrow \infty$, see Fig. \ref{fig:dissunif}. 
 \begin{figure}[h!]
 \begin{center}
 \includegraphics[height=6cm]{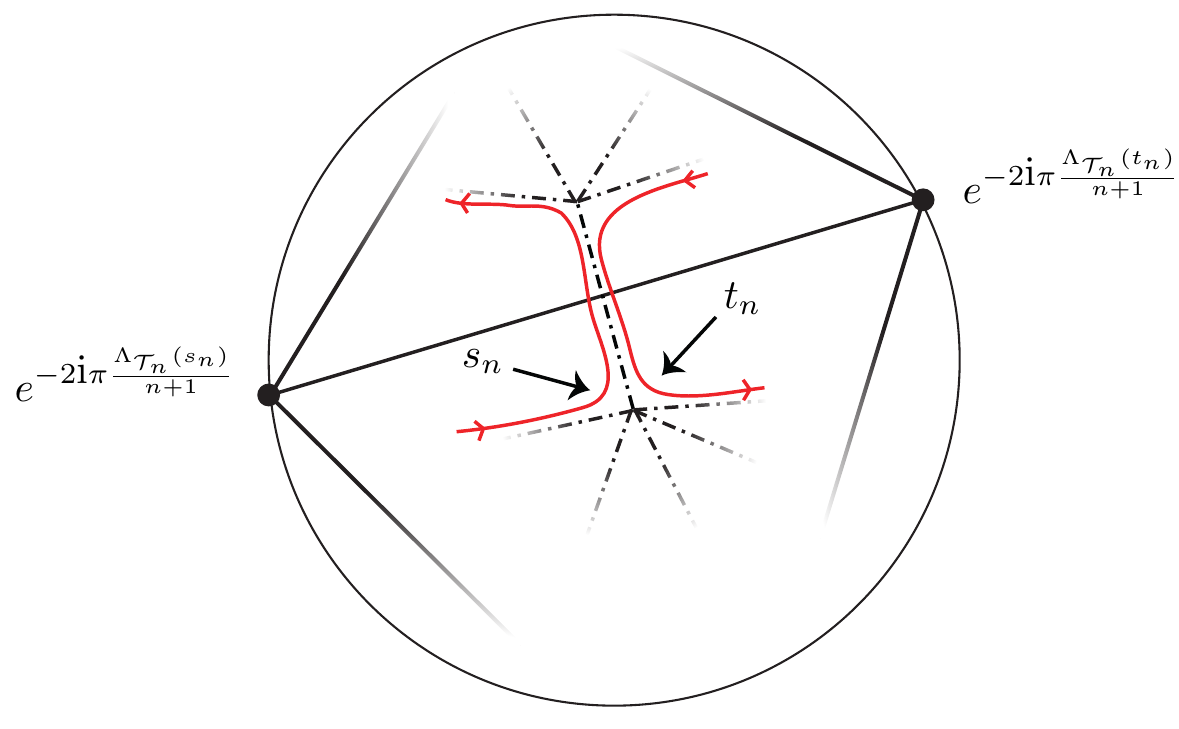}
 \caption{ \label{fig:dissunif} The arrows show the first visit time $s_n$ and last visit time $t_n$ of an edge of $ \mathcal {T}_n$.}
 \end{center}
 \end{figure}
 
Since the sides of $P_ {n+1}$, excepting the side connecting $1$ to $e^ {2 \textrm {i} \pi/ {(n+1)}}$, are in one-to-one correspondence with the leaves of $ \mathcal{T}_n$ we have (see Fig. \ref {fig:dissunif}):
 \begin{eqnarray*}\left[e^{-2\textrm{i} \pi \frac{\Lambda_{\t_n}(
s_n)}{n+1}}, e^{-2\textrm{i} \pi \frac{\Lambda_{\t_n}( t_n)}{n+1}}\right]&
\in& \mathcal{D}_n. \end{eqnarray*}
We refer to \cite {K:lam} for a complete proof. From Lemma
\ref{lem:repartitionleaves}, we can pass to the limit and obtain
$[e^{-2\textrm{i}\pi s},e^{-2\textrm{i}\pi t}] \subset \mathcal{D}_\infty$. 

Let us now  suppose that $ \mathbbm{e}(s)= \mathbbm{e}(t) =
\min_{[s\wedge t, s \vee t]} \mathbbm{e}$ and, moreover, there exists $r \in (s,t)$ such that
 $\mathbbm{e}(r) = \mathbbm{e}(s)$. Since local minima of Brownian motion are distinct, there exist two sequences of real numbers $( \alpha_n)_ {n \geq 1}$ and $( \beta_n)_ {n \geq 1}$ taking values in $[0,1]$ such that $ \alpha_n \rightarrow s$, $  \beta_n \rightarrow t$ as $ n \rightarrow \infty$ and such that for every $ n \geq 1$ and $r \in ( \alpha_n, \beta_n)$ we have $\mathbbm{e}(r) > \mathbbm{e}( \alpha_n) = \mathbbm{e}(\beta_{n})$. The preceding argument yields $[e^{-2\textrm{i}\pi  \alpha_n},e^{-2\textrm{i}\pi  \beta_n}] \subset \mathcal{D}_\infty$ for every $n \geq 1$. Since  $\mathcal{D}_\infty$ is closed, we conclude that $\mathcal{B} \subset \mathcal{D}_\infty$.

The reverse inclusion is obtained by making use of a maximality
argument. More precisely, it is easy to show that
$\mathcal{D}_\infty$ is a \emph{lamination}, that is a closed subset
of $ \overline{ \mathbb{D}}$ which can be written as a union of chords that do not intersect
each other inside $\D$. However, the Brownian triangulation $
\mathcal{B}$, which is also a lamination, is almost surely maximal for the
inclusion relation  among the set of all laminations of $ \overline{
\mathbb{D}}$, see \cite{LGP08}. It follows that  $
\mathcal{D}_\infty=\mathcal{B}$. This completes the proof of the
theorem.\endproof

\subsubsection{Application to the study of the number of intersections with a given chord}

We now explain how the ingredients of the previous proof can be used to
study the number of intersections of a large  dissection with
a given chord. For $ \a, \b \in [0,1]$, we denote by $I^{\a,\b}_n$ the
 number of intersections of $\mathcal{D}_n$  with the chord $[e^{-2\textrm{i} \pi \a},e^{-2 \textrm{i} \pi \b}]$, with the convention $I^{\a,\b}_n=0$ if $[e^{-2\textrm{i} \pi \a},e^{-2 \textrm{i} \pi \b}] \subset \mathcal {D}_n$.

\begin{prop}\label{prop:intersections}
 For $0 <\a < \b < 1$ we have
  \begin{eqnarray*}\frac{I^{\a,\b}_n}{\sqrt{n}} & \xrightarrow[n\to\infty]{(d)} &
 \frac{\mathbbm{e}(\b-\a)}{\sqrt{3\sqrt{2}-4}}, \end{eqnarray*}
 where $\mathbbm{e}$ is the normalized excursion of linear Brownian motion.
\end{prop}

\begin{proof}

For $1
\leq i \leq n$, denote by $l^n(i)$ the $i$-th leaf of $\t_n$ in the lexicographical order. Then,
for $1 \leq i<j \leq n$, the construction of the dual tree shows
that for every $s \in \left(\frac{i-1}{n+1},\frac{i}{n+1} \right)$
and $t \in \left(\frac{j-1}{n+1},\frac{j}{n+1} \right)$, $I^{s,t}_n$
is equal to the graph distance in the tree $\t_n$ between the leaves
$l^n(i)$ and $l^n(j)$.  Indeed, the edges of $ \t_n$ that intersect the chord $[e^{-2\textrm{i} \pi \a},e^{-2 \textrm{i} \pi \b}]$ are exactly the edges composing the shortest path between $l^n(i)$ and $l^n(j)$ in $ \mathcal{T}_n$. However, the situation is more complicated
when $e^{-2 \textrm{i} \pi s}$ or $e^{-2 \textrm{i} \pi t}$ coincides with a vertex of
$P_{n+1}$. To avoid these particular cases, we note that for
$\varepsilon,\varepsilon' \in (-\frac{1}{n+1},\frac{1}{n+1})$
we have
$$\left|I^{s+\varepsilon,t+\varepsilon'}_n-I^{s,t}_n\right| \leq 2 \Delta^{(n)},$$
where $\Delta^{(n)}$ is the maximal number of diagonals of $ \mathcal{D}_n$ adjacent to
a vertex of $P_{n+1}$. We claim that
\begin{eqnarray} \label{eq:majorationdelta1}
\frac{\Delta^{(n)}}{\sqrt{n}}  & \xrightarrow[n \to \infty]{(  \mathbb{P} )} & 0.
\end{eqnarray}
We leave the proof of the claim to the reader since a
(much) stronger result will be given in Theorem \ref{thm:vertexdegrees}. Let $0 < \alpha < \beta< 1$.
Set $i_n = \lfloor (n+1) \alpha \rfloor+1$ and $j_n= \lfloor (n+1)
\beta \rfloor+1$. Choose $n$ sufficiently large so that $j_n<n$. The
preceding discussion shows that
\begin{equation}\label{eq:majorationdelta2}
\left|I^{\a,\b}_n- \mathrm{d_{gr}}(l^n(i_n),l^n({j_n}))\right| \leq 2
\Delta^{(n)},\end{equation} where $ \mathrm{d_{gr}}$ stands for the graph
distance between two vertices in $\t_n$.  Now note that the graph
metric of the tree $\t_n$ can be recovered from  the contour
function  of $\t_n$, see \cite{DLG02}:  If $u_n,v_n$ are two vertices
of $\t_n$ such that the contour process reaches $u_n$ (resp. $v_n$)  at the instant $s_n$ (resp. $t_n$), then
 \begin{eqnarray} \label{eq:dgrc} \mathrm{d_{gr}}(u_n,v_n) \quad=\quad C_{\t_{n}}(s_{n})+C_{\t_n}(t_n)- 2 \inf_{u \in [s_n\wedge t_n, s_n\vee
t_n]} C_{\t_{n}} (u). \end{eqnarray}
If we choose $u_n = l^n({i_n})$ and $v_n=l^n({j_n})$ with respective first visit times $s_n$ and $t_n$,  Lemma \ref{lem:repartitionleaves} shows that $s_n/2\zeta(\t_n) \to \alpha$ and $t_n/2\zeta(\t_n) \to \beta$ in probability as $n \to \infty$. Consequently, using  \eqref{eq:majorationdelta1},(\ref{eq:majorationdelta2}), together with \eqref{cvcontour2}, and \eqref{eq:dgrc} we finally obtain
$$\frac{I^{\a,\b}_n}{\sqrt{n}} \quad \xrightarrow[n\to\infty]{(d)} \quad  (3\sqrt{2}-4)^{-1/2} \left( \mathbbm{e}( \alpha) + \mathbbm{e}(\beta)-2 \inf_{[ \alpha\wedge\beta, \alpha \vee \beta]}\mathbbm{e} \right).$$
To conclude, observe
 that by the re-rooting property of the Brownian excursion (see \cite[Proposition 4.9]{MM06}), the variable $\mathbbm{e}( \alpha) + \mathbbm{e}(\beta)-2 \inf_{[ \alpha\wedge\beta, \alpha \vee \beta]}\mathbbm{e}$ has the same distribution as
 $\mathbbm{e}(\beta-\alpha)$.
\end{proof}

\begin{rem}The preceding proof can be adapted easily to show the following functional convergence in distribution
\begin{eqnarray*}\left(\frac{I^{\a,\b}_n}{\sqrt{n}}\right)_{\begin{subarray}{c}0 \leq \a \leq 1\\
 0 \leq \b \leq 1 \end{subarray}} &
\xrightarrow[n\to\infty]{(d)} &  (3\sqrt{2}-4)^{-1/2} \cdot
 \left(\mathbbm{e}( \alpha) + \mathbbm{e}(\beta)-2 \inf_{[ \alpha\wedge\beta, \alpha \vee \beta]}\mathbbm{e}\right)_{\begin{subarray}{c}0 \leq \a \leq 1\\
 0 \leq \b \leq 1 \end{subarray}}.\end{eqnarray*}
\end{rem}

\subsection{Large uniform non-crossing trees}

In order to study large uniform non-crossing trees, the following
lemma will be useful. It roughly states that the location of a vertex in a non-crossing tree $ \mathscr{C}$ can be deduced from its location in the shape of  $\mathscr{C}$ up to an error that is bounded by its height. Recall that if $ \tau$ is a tree and $u \in \tau$, $k_u$ denotes the number of children of $u$.

\begin{lem} \label{lemnc} Let $\mathscr{C}$ be a non-crossing tree
with $n$ vertices and shape $S(\mathscr{C})=\tau$. Fix a vertex $u
\in \tau$ and let $a \in
\{0,1,\ldots,n-1\}$ be such that the  vertex  in
$\mathscr{C}$ corresponding to $u$ is $\exp(-2\mathrm{i}\pi{a}/{n})$. Then there
exists $i_{0} \in \{ 1, ... , k_{u}+1\}$ such that
$$ \left |a - \#\{ v \in \tau : v \prec ui_{0}\}\right | \leq |u|,$$
where $\prec$ stands for the strict lexicographical order on $ \mathcal{U}$.
\end{lem}

\proof  Let $u \in \tau \backslash \{\varnothing\}$. Consider the discrete geodesic path from $\varnothing$ to $u$ in $
\tau$ and its image $ \mathcal{L}$ in $ \mathscr{C}$. There exists
$1 \leq i_0\leq k_u+1$ such that, in $\mathscr{C}$, the first $i_0-1$
chidren of $u$ as well as their descendants are folded on the
left of $ \mathcal{L}$ (oriented from the root) and the rest of the descendants of $u$ are
folded on the right of $ \mathcal{L}$, see
Fig.\,\ref{preuvenc}. Now, consider the set $E= \{ 1, \exp(-2
\mathrm{i}\pi /n), ... , \exp(-2 \mathrm{i}\pi a/n)\}$ of all the
vertices of $P_n$ that are between $1$ and $\exp(-2 \mathrm{i}\pi a/n)$ in clockwise order. A
geometric argument (see Fig. \ref{preuvenc}) shows that if a vertex $x$ of $ \mathscr{C}$ belongs to $E$, then its corresponding vertex in the tree $\tau$  must belong to the set $\{ v \in \tau : v\prec ui_0\}$. On the
other hand, if $w \in \{ v \in \tau : v \prec ui_0\}$ and if,
moreover, $w$ is not a strict ancestor of $u$ in $ \tau$ then its corresponding vertex in $ \mathscr{C}$ belongs to $E$. Consequently, we have
$$ \# \{ v \in \tau : v \prec ui_0\} - |u|+1 \quad \leq\quad  \# E = a+1 \quad \leq\quad  \#\{ v \in \tau : v \prec ui_0\}.$$
The lemma follows (the case $u= \varnothing$ being trivial).
\begin{figure}[!h]
 \begin{center}
 \includegraphics[height=6cm]{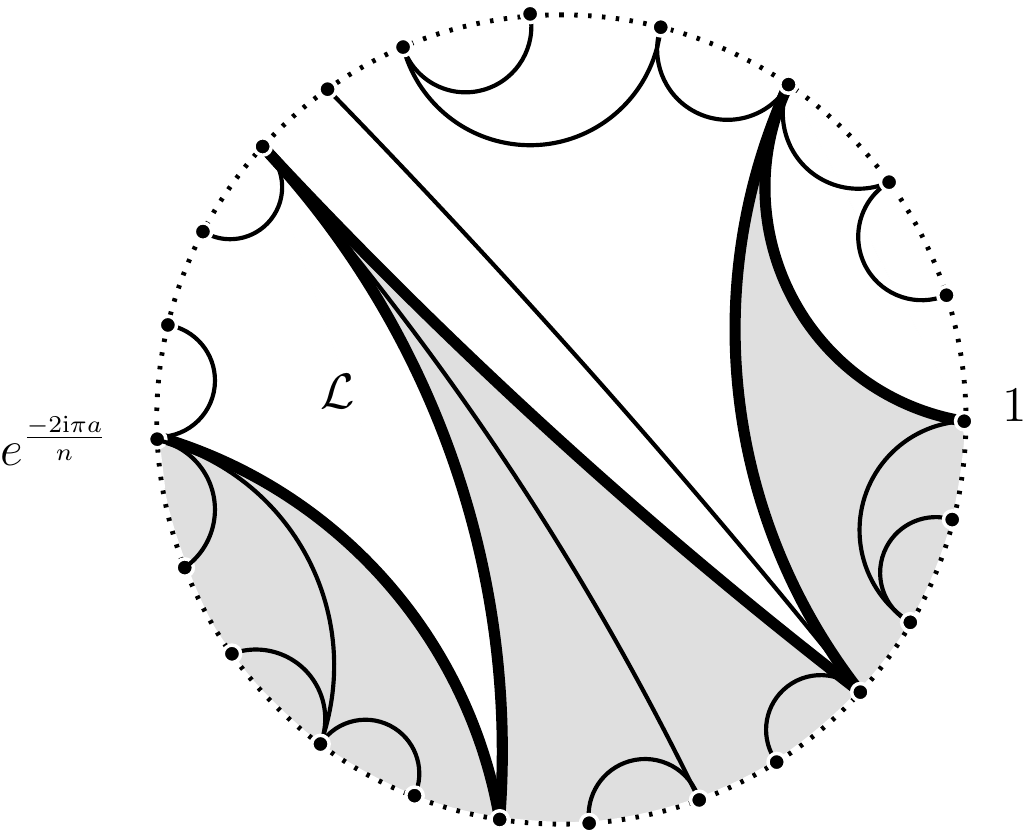}
 \caption{ \label{preuvenc} Illustration of the proof of Lemma \ref{lemnc}. We represent the non-crossing tree $ \mathscr{C}$ with curved chords for better visibility. The left-hand side of $ \mathcal{L}$ is in gray whereas its right-hand side is in white.}
 \end{center}
 \end{figure}

\endproof


For our purpose, it will be convenient to reinterpret this lemma
using the contour function. Fix a tree $\tau$, and define
$ \mathcal{Z}_{\tau}(j)$ as the number of distinct vertices of $\tau$
visited by the contour process of $\tau$ up to time $j$, for $ 0
\leq j \leq 2 \zeta(\tau)-2$. For technical reasons, we set $
\mathcal{Z}_{\tau}(j)= \zt$ for $j = 2 \zt-1 $ and $j=2 \zt$, and then extend
$ \mathcal{Z}_{\tau}(\cdot)$ to the whole segment $[0,2\zt]$ by linear
interpolation. Note that a vertex $u \in \tau$ with $k_{u}$ children
is visited exactly $k_{u}+1$ times by the contour function of
$\tau$, and that if $t^{(0)}, ... , t^{(k_{u})}$ are these times, then for
every $i_{0} \in \{ 0, ... , k_{u}\}$ we have
 \begin{eqnarray} \label{etoile} \#\{ v \in \tau : v \prec u(i_{0}+1)\} = \mathcal{Z}_{\tau} \left (t^{(i_{0})} \right).  \end{eqnarray}

The idea of the proof of Theorem \ref{thm:brownianlimit} part (ii) is the following. Let $ \mathscr{C}$ be a large non-crossing tree with shape $S(\mathscr{C})$. Pick a vertex $u \in S( \mathscr{C})$ corresponding to the point $\exp(-2ia/n)$ in $ \mathscr{C}$. The goal is to recover $a$ (with error at most $o(n)$) from the knowledge of $ S( \mathscr{C})$ and $u$. Assume that $u$ is a leaf of $ S(\mathscr{C})$. Then $u$ is visited only once by the contour process, say at time $t_u$. By Lemma \ref{lemnc} the quantity $|a- \mathcal{Z}_{S(\mathscr{C})}(t_u)|$ is less than the height of the tree $ S(\mathscr{C})$ which is small in comparison with $n$ by \eqref{cvcontour}.  Hence $a$ in  is known up to an error $o(n)$. We will see that the control by the leaves of $ S(\mathscr{C})$ is sufficient for proving the convergence towards the Brownian triangulation.

The next lemma is an analogous to Lemma \ref{lem:repartitionleaves}.

\begin{lem} We have
 \begin{eqnarray}    \label{cvcontourbis} \sup_{0 \leq t \leq 1} \left|
 \frac{\mathcal{Z}_{ \mathscr{T}_{n}}(2nt)}{n}-t\right|
 &\xrightarrow[n\to\infty]{(\P)}& 0. \end{eqnarray}
\end{lem}
\proof This is a standard consequence of \eqref{cvcontour}, see e.g.\,\cite[Section 1.6]{LG06}.
\endproof

\proof[Proof of Theorem \ref{thm:brownianlimit} part (ii)]

Similarly to the proof of part $(i)$ of the theorem, we can
apply Skorokhod's theorem and assume that the convergences
\eqref{cvcontour} and \eqref{cvcontourbis} hold almost surely.  We fix $ \omega$ such that both convergences \eqref{cvcontour}  and \eqref{cvcontourbis} hold for this value of $ \omega$. Up to extraction, we thus suppose that $(\mathscr{C}_{n})_{n \geq 1}$ converges towards a compact subset $\mathscr{C}_\infty$ of $\Db$ and we aim at showing that
$\mathscr{C}_\infty=\mathcal{B}$.

We first show that $\mathcal{B} \subset \mathscr{C}_\infty$. Fix
$0<s<t<1$ such that $ \mathbbm{e}(s)= \mathbbm{e}(t) =
\min_{[s\wedge t, s \vee t]} \mathbbm{e}$ and assume furthermore that $
\mathbbm{e}(r) > \mathbbm{e}(s)$ for $r \in (s,t)$. Let us show that
$[e^{-2\textrm{i} \pi s},e^{-2\textrm{i} \pi t}] \subset \mathscr{C}_\infty$. To this
end, we fix $\varepsilon>0$ and show that
 $ [e^{-2\textrm{i} \pi s},e^{-2\textrm{i} \pi t}] \subset \mathscr{C}_{n}^{(6
 \varepsilon)}$ for $n$ sufficiently large (recall that $X ^{(\varepsilon)}$ is the $ \varepsilon$-enlargement of a
closed subset $ X \subset \overline{ \mathbb{D}}$). Using the convergence \eqref{cvcontour}, for every $n$ large enough, one
can find integers $0 \leq s_{n}< t_{n} \leq 2n-2$ such that if
$u_{n}$ (resp.\,$v_{n}$) denotes the vertex of $ \mathscr{T}_{n}$
visited at time $s_{n}$ (resp.\,$t_{n}$) by the contour process, the
following three properties are satisfied:
\begin{itemize}
\item $s_{n}/2n \to s$, $t_{n}/2n \to t$,
\item $u_{n}$ and $v_{n}$ are leaves in $ \mathscr{T}_{n}$,
\item   for every vertex $w_{n}$ in $\llbracket u_{n},v_{n}\rrbracket$ (the discrete geodesic path between $u_{n}$ and $v_{n}$
in $ \mathscr{T}_{n}$) and every visit time $r_{n}$ of $w_{n}$ by
the contour process, we have
 \begin{eqnarray} \label{eq:geodproche}\min\left( \left|\frac{r_{n}}{2n} - s\right|, \left|\frac{r_{n}}{2n} - t \right|\right)
 &\leq& \varepsilon. \end{eqnarray}
\end{itemize}
For the second property, we can for instance use the fact that local
maxima of $ \mathbbm{e}$ are dense in $[0,1]$. We now claim that
under these assumptions, for $n$ large enough, the image $
\mathcal{L}_n$ in $ \mathscr{C}_{n}$ of the discrete geodesic path $\llbracket
u_{n},v_{n}\rrbracket$ in $ \mathscr{T}_{n}$ lies within Hausdorff distance
$6\varepsilon$ from the line segment $[e^{-2\textrm{i} \pi s},e^{-2\textrm{i} \pi t}]$.

Indeed, let $w_{n} \in \llbracket u_{n},v_{n} \rrbracket$ and let $a_{n} \in \{0,1, \ldots , n-1\}$ such that the vertex of $ \mathscr{C}_{n}$ corresponding to $w_{n}$ is  $z_{n} =\exp( -2 \mathrm{i}\pi a_{n}/n)$. Applying Lemma \ref{lemnc} to $u = w_{n}$ and using \eqref{etoile} we can find a time $r_{n}$ at which the contour process is at $w_{n}$ and such that  \begin{eqnarray*}  |a_{n} -
\mathcal{Z}_{\mathscr{T}_{n}}(r_{n})| &\leq& |w_{n}|.  \end{eqnarray*} By the
convergence \eqref{cvcontourbis} and the bound \eqref{eq:geodproche},
 there exists an integer $N \geq 1$, independent of the choice
of $w_n$, such that for $n \geq N$
\begin{eqnarray*} \min \left( \left|\frac{\mathcal{Z}_{\mathscr{T}_{n}}(r_{n})}{n}-s \right|
,\left|\frac{\mathcal{Z}_{\mathscr{T}_{n}}(r_{n})}{n}-t
\right|\right) &\leq& 2\varepsilon.
\end{eqnarray*} On the other hand, thanks to the convergence \eqref{cvcontour} we have
$$\frac{1}{n}\displaystyle \sup_{u \in \mathscr{T}_{n}
}|u|=\frac{1}{\sqrt{n}} \frac{1}{\sqrt{n}}  \displaystyle \sup_{0
\leq t \leq 1} C_{\mathscr{T}_{n}}(2nt) \quad
\xrightarrow[n\to\infty]{}  \quad 0 .$$ We
conclude that there exists an integer $N' \geq 1$, independent of the choice of
$w_n$, such that for $n \geq N'$, we have $\min ( |\frac{a_{n}}{n}-s |
,|\frac{a_{n}}{n}-t|) \leq 3 \varepsilon$ and thus that $ \min\left(\left|z_{n}-e^{-2
\mathrm{i}\pi s}|,|z_{n}-e^{-2 \mathrm{i}\pi t}\right|\right) \leq 6
\varepsilon$. It follows that for large $n$, we have \begin{equation}
\label{firstinclu} \mathcal{L}_{n}\subset [e^{-2\textrm{i} \pi s},e^{-2\textrm{i} \pi
t}]^{(6\varepsilon)} .\end{equation} On the other hand, $u_{n}$ and
$v_{n}$ are leaves, so they are visited at a unique time by the contour
process. By the same arguments, for every $n$ sufficiently large, we deduce that their images $\a_{n}$
and $\b_{n}$ in $ \mathscr{C}_{n}$ satisfy $|\a_{n}
-e^{-2\textrm{i} \pi s}| \leq 6\varepsilon$ and $|\b_{n} -e^{-2\textrm{i} \pi t}|
\leq 6\varepsilon$. Consequently, since $\mathcal{L}_{n}$ is a finite union 
of line segments connecting $\alpha_{n}$ to $\beta_{n}$, 
 we deduce from \eqref{firstinclu} that for every $n$ sufficiently large
enough
  \begin{equation} \label{finalnc}[e^{-2\textrm{i} \pi s},e^{-2\textrm{i} \pi t}]  \subset  \mathcal{L}_{n}^{(6\varepsilon)}
  \subset \mathscr{C}_{n}^{(6\varepsilon)}.
  \end{equation}
The case when there exists $r \in (s,t)$ such that $ \mathbbm{e}(s)=
\mathbbm{e}(r)= \mathbbm{e}(t)$ (this $r$ is then a.s.\,unique by
standard properties of the Brownian excursion) is treated exactly as in the proof of the first assertion of this theorem. We conclude that
$\mathcal{B} \subset \mathscr{C}_\infty$. The reverse inclusion is obtained by making use of a maximality argument, see part $(i)$.\endproof

\subsection{Universality of the Brownian triangulation and applications}

\label {sec:universality}

The convergence in distribution of random compact subsets towards
the Brownian triangulation yields information on their asymptotic
geometrical properties that are preserved under the Hausdorff convergence.
 Let us give an example of application of this fact.
 
 Let $ \chi_n$ be a random configuration on the vertices of $P_n$, that is a random closed subset made of  line segments connecting some of the vertices of the polygon. Assume that $\chi_n$ converges in distribution towards $ \mathcal{B}$ in the sense of the Hausdorff metric. Let $ \mathrm{diag}( \chi_n)$ be the Euclidean
  length of the longest diagonal of $ \chi_n$. Then, as $n \to \infty$, the law of $ \mathrm{diag}( \chi_n)$ converges in distribution towards the length of the longest chord of the Brownian triangulation, given by: $$\frac{1}{\pi} \frac{3x-1}{x^2(1-x)^2 \sqrt{1-2x}} \mathbf{1}_{\frac{1}{3} \leq x \leq \frac{1}{2}} \mathrm{d}x.$$
 This distribution has been computed in \cite{Ald94b}  (see also \cite{DFHN99}). The preceding convergence follows from the fact that the length of the longest chord  is a continuous function of configurations for the Hausdorff metric.  Similar limit theorems hold for a large variety of other functionals, such as the area of the face with largest area, etc.

\medskip

  It is plausible that many other uniformly distributed
 non-crossing configurations (see \cite{FN99})  converge towards the Brownian triangulation in the Hausdorff sense. We give here a few instances of this phenomenon.

\paragraph{Dissections with constrained face degrees.} 
\begin{thm}\label{prop:gen}Let $\mathcal{A}$ be a non-empty subset of $\{3,4,5,\ldots\}$. Let $\mathbf{D}^{(\mathcal{A})}_n$ be the set of all dissections of $P_{n+1}$ whose face degrees all belong to the set $\mathcal{A}$. We restrict our attention to the values of $n$ for which $ \mathbf{D}_{n}^{( \mathcal{A})} \ne  \varnothing$. \begin{enumerate}[(i)]
\item There exists a probability distribution $\nu_{\mathcal{A}}$ on $\N$ such that if $\sigma_{ \mathcal{A}}^2$ denotes the variance of $\nu_{\mathcal{A}}$, we have 
$$ \#\mathbf{D}^{(\mathcal{A})}_{n-1} \quad  \mathop{\sim}_{n \rightarrow \infty} \quad  \sqrt{\frac{\nu_{\mathcal{A}}(2)^4 \nu_{\mathcal{A}}(0)^3} {2 \pi
\sigma_{\mathcal{A}}^2}} \cdot \frac{n^{-3/2}}{\left(\nu_{\mathcal{A}}(2)
\nu_{\mathcal{A}}(0)\right)^{n}} .$$
\item Let
$\mathcal{D}^{(\mathcal{A})}_n$ be  uniformly distributed over
$\mathbf{D}^{(\mathcal{A})}_n$. Then $\mathcal{D}^{(\mathcal{A})}_n$
converges towards the Brownian triangulation.
\end{enumerate}
\end{thm}

Note that the case $\mathcal{A}=\{3\}$ corresponds to uniform triangulations
and the case $\mathcal{A}=\{3,4,5,\ldots\}$ corresponds to uniform
dissections.

\begin{proof}The proof of this  statement goes along
the very same lines as the proofs of Corollary \ref{prop:comptage} and Theorem \ref {thm:brownianlimit} (i) by noticing that the dual tree $\phi(\mathcal{D}^{(\mathcal{A})}_n)$
is a Galton-Watson tree conditioned on having $n$ leaves for a certain finite variance offspring
distribution $\nu_{\mathcal{A}}$. More precisely, if we denote the set $ \{ a-1 : \, a \in A\}$ by $\mathcal{A}-1$,   let $c_{\mathcal{A}} \in
(0,1)$ be the unique real number in $(0,1)$ such that
$$\sum_{i \in \mathcal{A}-1} i c_{\mathcal{A}}^{i-1}=1.$$
Then $\nu_{\mathcal{A}}$ is defined by
$$\nu_{\mathcal{A}}(0)=1-\sum_{i \in \mathcal{A}-1} c_{\mathcal{A}}^{i-1}, \qquad\nu_{\mathcal{A}}(i)=c_{\mathcal{A}}^{i-1}  \textrm{  for }i \in
\mathcal{A}-1.$$ Note that $\nu_{\mathcal{A}}(2)=c_{\mathcal{A}}$ and that $\nu_{\mathcal{A}}$ automatically has a finite variance $\sigma_{ \mathcal{A}}^2>0$.
\end{proof}

\paragraph{Non-crossing graphs.}
A \emph{non-crossing graph} of $P_n$ is a graph drawn on the plane, whose vertices are the
vertices of $P_n$ and whose edges are non-crossing line segments. Let $\mathcal{G}_n$ be uniformly distributed over the set of all non-crossing graphs of $P_n$. Note that $\mathcal{G}_n$ can be  seen as a
compact subset of $ \overline{ \mathbb{D}}$. Then 
$\mathcal{G}_{n} $ converges in distribution towards  the Brownian triangulation.

This fact easily follows from the convergence of uniform dissections towards the Brownian triangulation. Indeed, if $\mathcal{G}$ is a non-crossing graph  of $P_n$,
let $\psi(\mathcal{G})$ be the compact subset of $\overline{
\mathbb{D}}$ obtained from $\mathcal{G}$ by adding the sides of
$P_n$. 
As noticed at the end of Section 3.1 in \cite{FN99},
$\psi(\mathcal{G})$ is a dissection, and every dissection has $2^n$
pre-images by $\psi$. It follows that the random dissection
$\psi(\mathcal{G}_n)$ is a uniform dissection of $P_{n}$. 
The conclusion follows, since the Hausdorff distance
between $\psi(\mathcal{G}_{n})$ and $\mathcal{G}_{n}$ tends to $0$.

\paragraph{Non-crossing partitions and non-crossing pair partitions.} A non-crossing partition of $P_n$ is a partition of the vertices of $P_n$ (labeled by the set $ \{1,2, \ldots, n \}$) such that the convex hulls of its blocks are pairwise disjoint (see Fig.  \ref {fig:partitions} where the partition $ \{ \{ 1,2,4,8\},\{ 3\},\{ 5\},\{6,7 \}\}$ is represented). A non-crossing pair-partition of $P_n$ is a non-crossing partition of $P_n$ whose blocks are all of size $2$ (see Fig. \ref {fig:partitions} where the pair-partition $  \{ \{ 1,16\}, \{2,3 \},\{4,7 \},\{5,6 \},\{8,15 \}, \{9,10 \},\{11,14 \}$, $\{12,13 \}\}$ is represented).  
\begin{figure}[!h]
 \begin{center}
 \includegraphics[height=5cm]{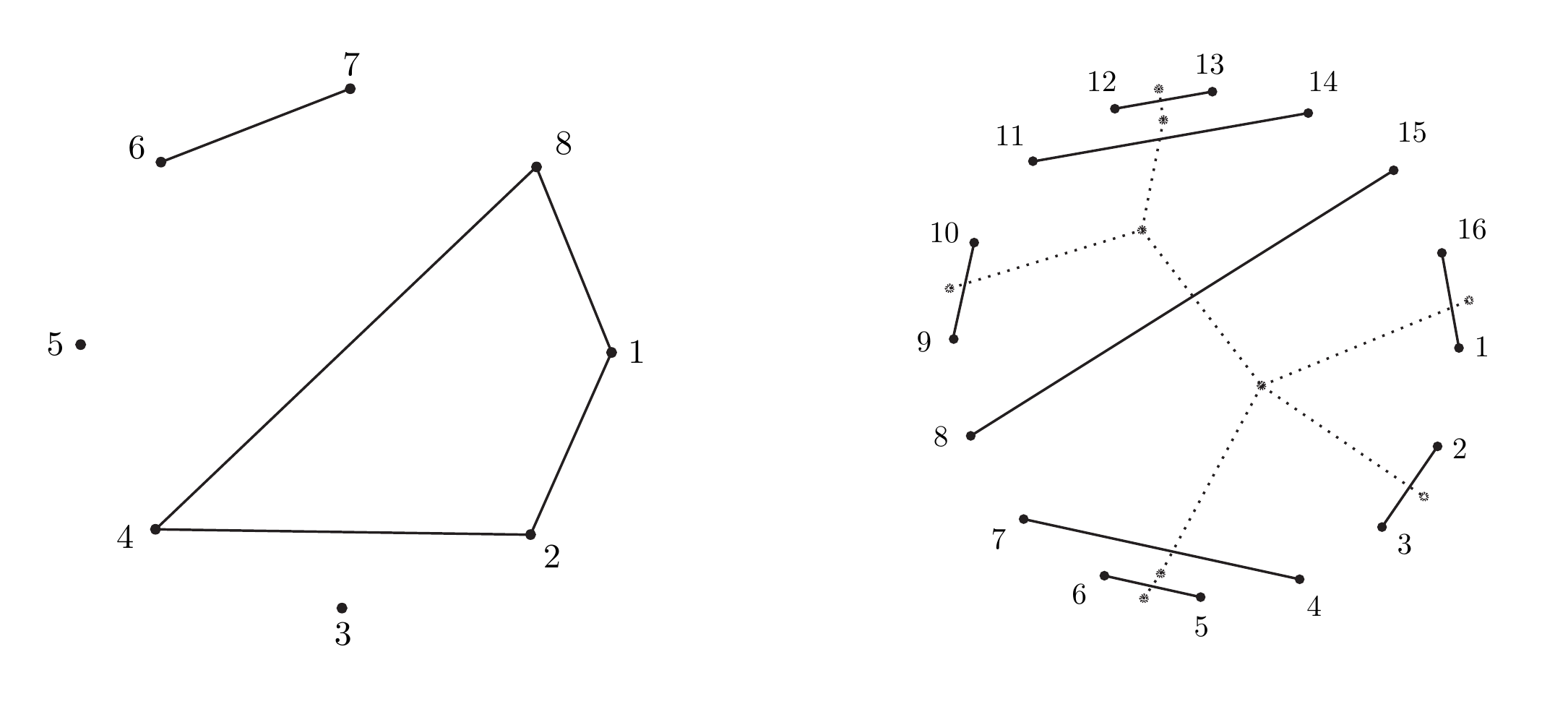}
 \caption{ \label{fig:partitions} A non-crossing partition of $P_8$ and a non-crossing pair-partition of $P_{16}$ together with its dual tree.}
 \end{center}
 \end{figure}
 
Let $ \mathcal{P}^{(2)}_n$ be a uniformly distributed random variable on the set of all non-crossing pair-partitions of $P_{2n}$,  seen as a compact subset of $ \overline{ \mathbb{D}}$. Then $\mathcal{P}^{(2)}_{n} $ converges towards  the Brownian triangulation.

To establish this fact, we rely once again on a coding of $\mathcal{P}^{(2)}_n$ by a critical Galton-Watson tree. One easily sees that the dual tree of $\mathcal{P}^{(2)}_n$ (see Fig. \ref {fig:partitions}) is a uniform tree with $n$ edges, which is also well known to be a Galton-Watson tree with geometric offspring distribution, conditioned on having $n$ edges. One can then show the convergence of $\mathcal{P}_{n} $ towards the Brownian triangulation using the same methods as in the case of uniform dissections. Details are left to the reader.

\medskip

Let us now discuss non-crossing partitions.   Let $ \mathcal{P}_n$ be a uniformly distributed random variable on the set of all non-crossing partitions of $P_{n}$,  and view $ \mathcal {P}_n$ as a random compact subset of $ \overline{ \mathbb{D}}$. Then $\mathcal{P}_{n} $ converges towards  the Brownian triangulation.

This  follows from the convergence of non-crossing pair-partitions. Indeed, given a non-crossing pair-partition of $P_{2n}$, we get a non-crossing partition of $P_{n}$ by identifying the $n$ pairs of vertices of the form $(2i-1,2i)$ for $1 \leq i \leq n$ (see Fig. \ref {fig:partitions} where the non-crossing partition of $P_8$ is obtained by contraction of vertices from the non-crossing pair-partition of $P_ {16}$). This identification gives a bijection between non-crossing partitions of $P_n$ and non-crossing pair partitions of $P_ {2n}$ and the desired result easily follows.

\begin{figure}[!h]  \begin{center}
\includegraphics[height=3cm]{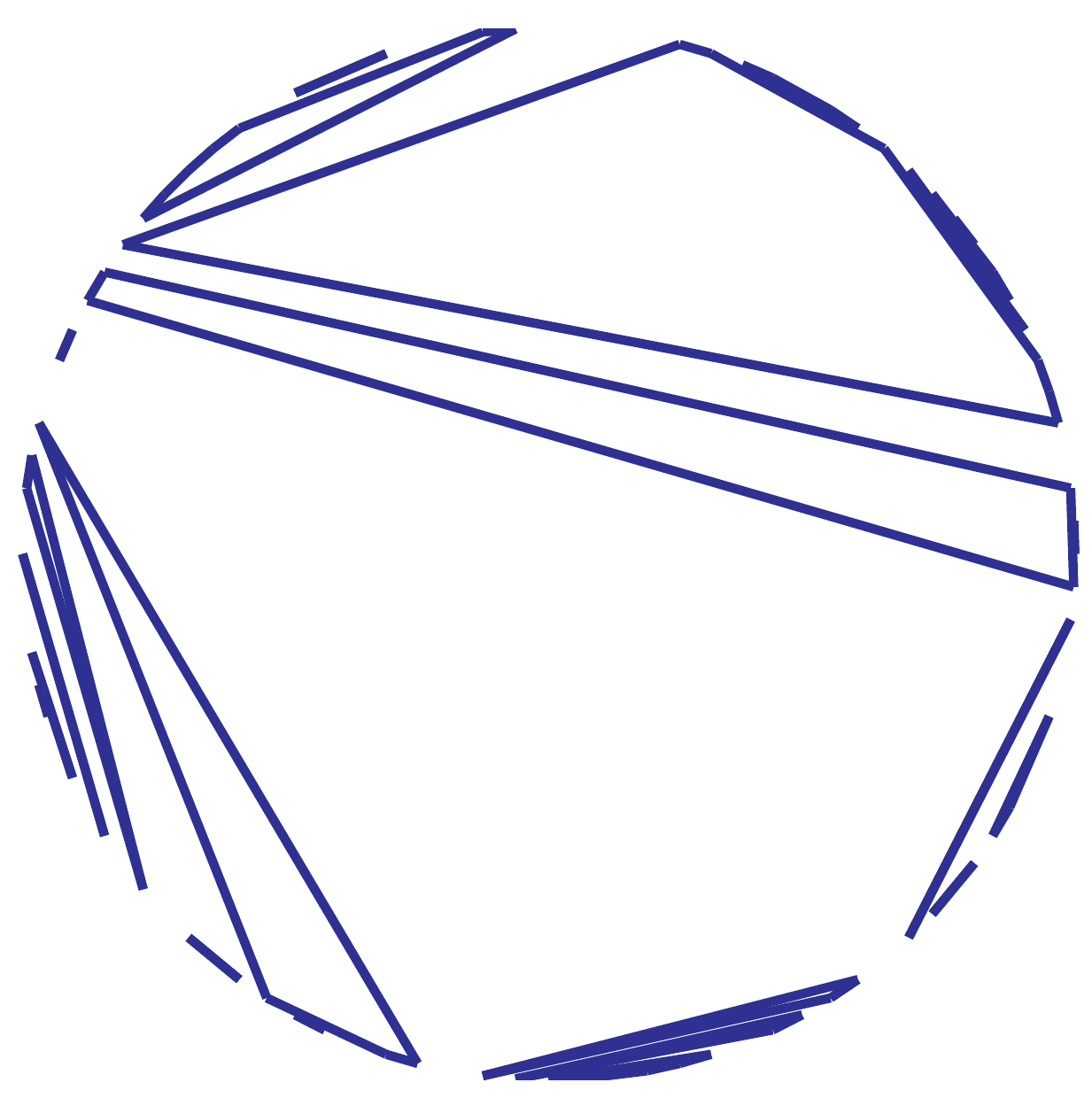} 
\hspace{0.5cm}
\includegraphics[height=3cm]{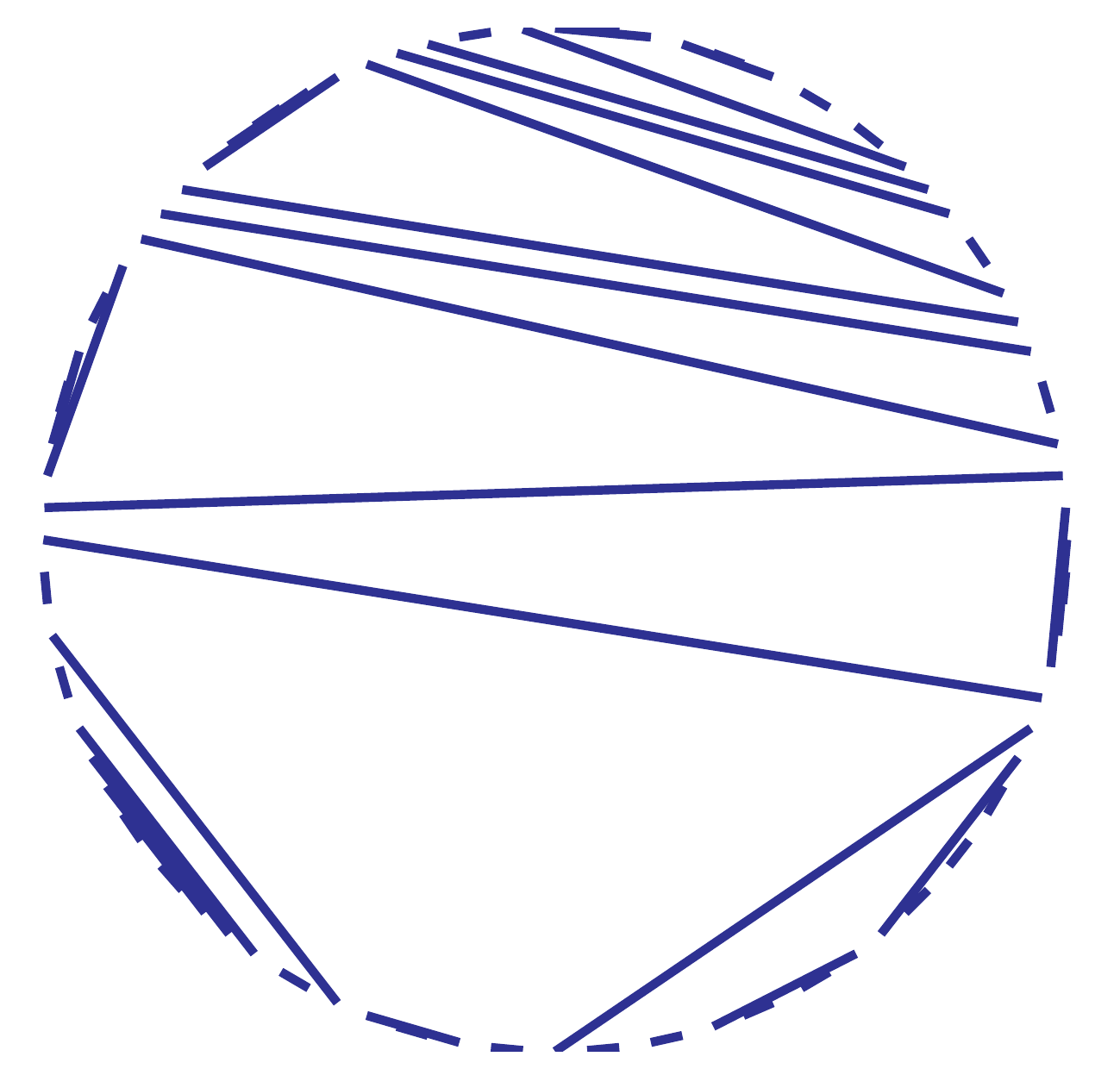}
\caption{
\label{fig2} Uniform non-crossing partition and pair-partition of $P_ {100}$. }  \end{center}
\end{figure}

 At first sight, it may seem mysterious that Galton-Watson trees appear behind many different models of uniform non-crossing configurations. In \cite{FN99}, using suitable parameterizations, Flajolet and Noy manage to find a Lagrange inversion-type implicit equation for the generating functions of these configurations. Generating functions verifying a Lagrange inversion-type implicit equation are those of simply-generated trees, which are very closely related to Galton-Watson trees (see \cite[Section 2.1]{Aldous2}). This explains why Galton-Watson trees are hidden behind various models of uniformly distributed non-crossing configurations.

\section{Graph-theoretical properties of large uniform dissections}
 In this section, we study graph-theoretical properties of large
dissections using the Galton-Watson tree structure
identified in Proposition \ref{prop:unif}. Let us stress that as in
Proposition \ref{prop:gen}, all the results contained in this
section can be adapted easily to uniform dissections constrained on
having all face degrees in a fixed non-empty subset of $\{3,4,\ldots\}$
(and in particular to uniform triangulations).

As previously, $\mathcal{D}_{n}$ is a uniformly distributed dissection of $P_{n+1}$ and $ \mathcal{T}_{n}$ denotes its dual tree with $n$ leaves. We start by recalling the definition and the construction of the
so-called critical Galton-Watson tree conditioned to survive.


\subsection{The critical Galton-Watson tree conditioned to survive}
 \label{cgwi}

If $\tau$ is a tree and $k$ is a nonnegative integer, we let $[\tau]_{k}=  \{u \in \tau : \, |u| \leq k\}$ denote the tree obtained from $\tau$ by keeping the
vertices in the first $k$ generations. Let $ \mathbbm{\xi}=(\xi_{i})_{i\geq 0}$ be an
offspring distribution with $\xi_1 \neq 1$ and $\sum i \xi_{i} =1$.
We denote by $T_{n}$ a Galton-Watson tree with offspring distribution $\xi$ conditioned on
having height at least $n \geq 0$. Kesten \cite[Lemma 1.14]{Kes86} showed that for every $k \geq 0$, we
have the following convergence in distribution
  \begin{eqnarray*} \left[T_{n}\right]_{k} &\xrightarrow[n\to\infty]{(d)}& \left[T_{\infty}\right]_{k}, \end{eqnarray*} where $T_{\infty}$ is a random infinite plane tree called the critical $\xi$-Galton-Watson tree conditioned to survive.

We denote the law of the $\xi$-Galton-Watson tree conditioned to survive by $ \widehat{ \mathbb{P}}_{\xi}$. Let us describe this law (see
\cite{Kes86,LP10}). We let $\overline{\xi}$ be the size-biased
distribution of $\xi$ defined by $ \overline{\xi}_{k} = k \xi_{k}$
for $k \geq 0$. The random variable $T_{\infty}$ distributed
according to $ \widehat{ \mathbb{P}}_{\xi}$ is described as follows. Let $(D_i)_ {i \geq 0}$ be a sequence of i.i.d. random variables distributed according to $
\overline{\xi}$. Let also $(U_i)_ {i \geq 1}$ be a sequence of random variables such that, conditionally on $(D_i)_ {i \geq 0}$, $(U_i)_ {i \geq 1}$ are independent and  $U_{k+1}$ is uniformly distributed over $\{ 1, 2, ... , D_{k}\}$ for every $k \geq 0$. The tree $T_{ \infty}$ has a unique spine, that is a unique infinite path
$(\varnothing, U_{1},U _1 U_{2}, U_1 U_2 U_{3}, ...) \in {\mathbb{N}^*}^{ \mathbb{N}^*}$ and the degree of $U_{1}U_{2}... U_{k}$ is $D_k$. Finally,
conditionally on $(U_i)_ {i \geq 1}$ and $(D_i)_ {i \geq 0}$ all the remaining subtrees are independent  ${\xi}$-Galton-Waton trees.
\begin{figure}[!h]
  \begin{center}
  \includegraphics[height=8cm]{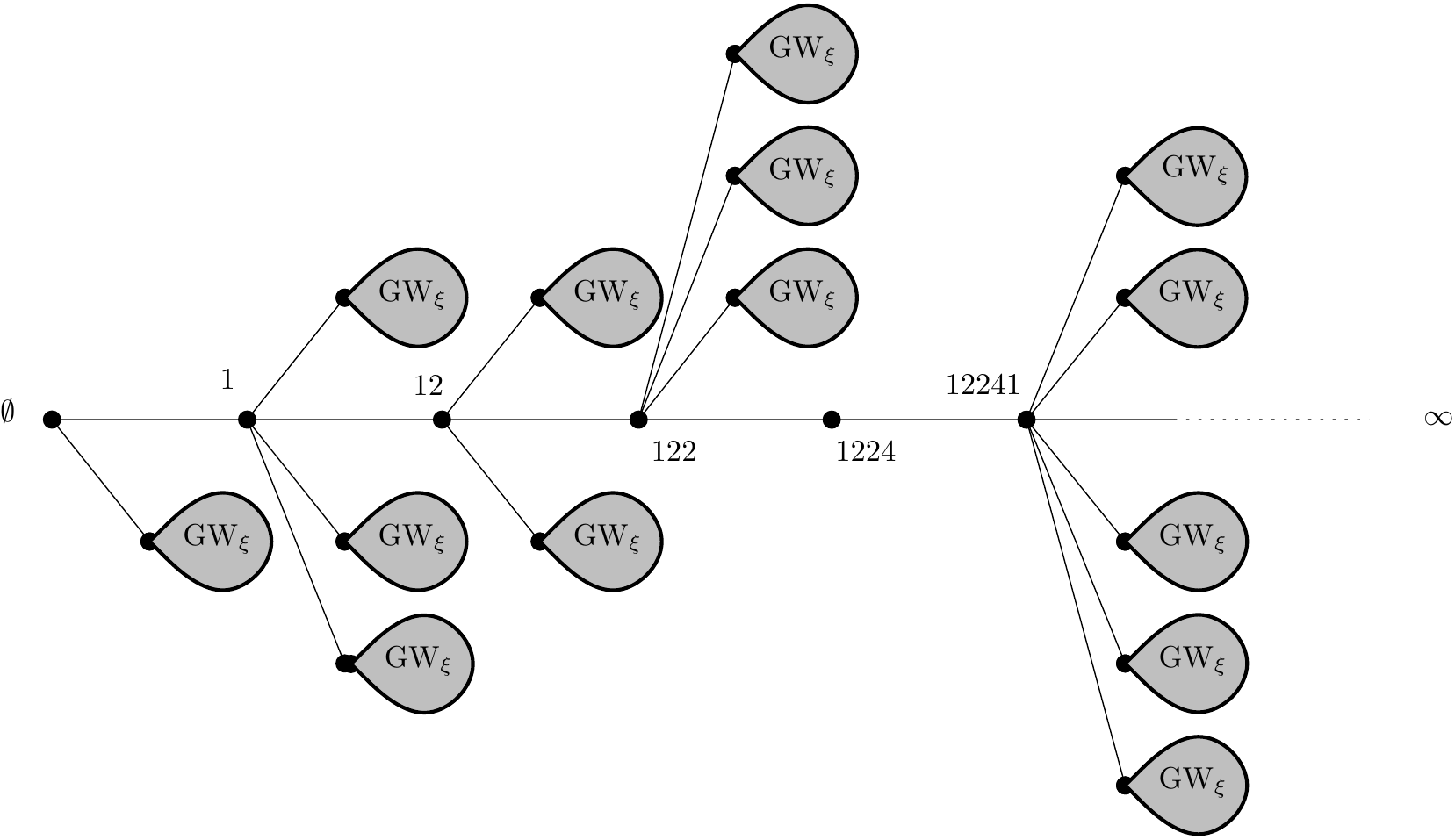}
  \caption{\label{critical} An illustration of $T_{\infty}$ under $ \widehat{\mathbb{P}}_\xi$.}
  \end{center}
  \end{figure}

The critical Galton-Watson tree conditioned to survive also
arises in other conditionings of Galton-Watson trees. Recall that $\t_{n}$ is a $ \mu$-Galton Watson tree conditioned on having $n$ leaves.

\begin{thm}\label{thm:cvlocale} 
For every $k \geq 0$, we have the convergence in distribution
  \begin{eqnarray*} \left[\t_{n}\right]_{k} & \xrightarrow[n\to\infty]{(d)}& \left[  \mathcal{T}_{\infty} \right]_{k},  \end{eqnarray*} where $ \mathcal{T}_{\infty}$ is the critical Galton-Watson tree with offspring distribution $\mu$ conditioned to survive.
  \end{thm}
  
\begin{rem} Theorem \ref{thm:cvlocale} is true when $ \mu$ is replaced by any finite variance offspring distribution $\nu$ such that $ \P_ { \nu} \left(  \lt =n\right) >0$ for every $n$ large enough.
\end{rem}

\proof This follows from another description of the law $ \widehat{
\mathbb{P}}_{\mu}$: If $ \tau_0$ is a plane tree and $k
\geq 0$ is an integer, we denote by $L_{k}(\tau_0)$ the number of
individuals of $\tau_0$ at height exactly $k$. Then we have from
\cite[Lemma 1.14]{Kes86}:
 \begin{eqnarray*}\widehat{\P}_\mu([\tau]_k=\tau_0)&=& L_k(\tau_0) \Pmu([\tau]_k=\tau_0).  \end{eqnarray*}
 Fix an integer $k \geq 1$ as well as a tree $\tau_0 \in \T$ of height
 $k$. In order to prove the theorem, it is thus sufficient to show
that  \begin{eqnarray*}\label{eq:cvlocale}\Pmu([\tau]_k=\tau_0 \, | \,
\lt=n) & \displaystyle \mathop{\longrightarrow}_{n \rightarrow
\infty} & L_k(\tau_0) \Pmu([\tau]_k=\tau_0). \end{eqnarray*} Denote by $q$ the number of leaves of $\tau_0$ that have a height strictly smaller than $k$. By the branching property of Galton-Watson trees, we have
$$ \Pmu([\tau]_k=\tau_0 \, | \, \lt=n)= \Pmu([\tau]_k=\tau_0)
\frac{\P_{\mu,L_k(\tau_0)}(\lambda(\bf)=n-q)}{\Pmu(\lt=n)},$$
where we recall that $ \mathbb{P}_{\mu,i}$ denotes the probability
distribution of a forest of $i$ independent Galton-Watson trees with law $ \mathbb{P}_{\mu}$. Since $q$ and $L_k(\tau_0)$ are fixed, it is thus
sufficient to show that for a fixed integer $i \geq 1$, as
$n\rightarrow \infty$
$$ \P_{\mu,i}(\lambda(\bf)=n) \quad\mathop{\sim}_{n \to \infty} \quad i \times
\Pmu(\lt=n).$$
This follows from the next lemma, which we state in a more general form than needed here in view of further applications.\endproof

\begin{lem} \label{ui}There exists $ \varepsilon>0$ such that if $(i_n)_ {n \geq 1}$ is a sequence of positive integers such that $i_n \leq n^ \varepsilon$ for every $n \geq 1$, we have  
 $$ \P_{\mu,i_{n}}(\lambda( \mathfrak{f})=n)  \quad\mathop{\sim}_{n \to \infty} \quad i_{n} \cdot 
\Pmu(\lt=n).$$

\end{lem}

\proof We show that the conclusion of the lemma holds for any $\varepsilon \in (0, 1/9)$.  To simplify notation, we set $p_{k}:=\mathbb{P}_{\mu}(\lambda(\tau)=k)$ for every integer $k \geq
1$ and write $i=i_n$ in the proof. By the definition of $ \mathbb{P}_{\mu,i}$, we have
 \begin{eqnarray*} \mathbb{P}_{\mu, i}(\lambda( \mathfrak{f}) = n) &=& \sum_{k_{1} + ... + k_{i} = n} \prod_{j=1}^{i} p_{k_{j}}.  \end{eqnarray*}
We will show that when $n$ is large, the main contribution in the
previous sum is obtained when the indices $k_{1}, ... , k_{i}$ are such that
only one of them is of order $n$ and the others are small in
comparison. Let $A \geq 1$. Firstly, notice
that at least one of the indices $k_{1}, ... , k_{i}$ is larger than
$n/i$. Secondly, let us evaluate the contribution of the sum when
$k_{1} \geq n/i$ and $k_{2}$ is larger than $A$. By Lemma \ref{estimee},
$$ \sum_{\begin{subarray}{c}k_{1}+ \cdots + k_{i}=n \\ k_{1} \geq n/i\\ k_{2} \geq  A \end{subarray}}
\prod_{j=1}^{i} p_{k_{j}} \quad \leq \quad \sup_{ k_{1} \geq n/i}p_{k_{1}}
 \cdot \left(\sum_{k_{2} \geq A}p_{k_{2}}\right) \cdot \prod_{j=3}^i \left(\sum_{k_{j}=1}^\infty p_{k_{j}}\right)
  \quad \leq \quad C (n/i)^{-3/2}A^{-1/2},$$
for some constant $C>0$ which is independent of $n$ and $i$. Hence, provided that $A<n ^ {1- \varepsilon}$, \begin{eqnarray} \left | \left(\sum_{k_{1} + \cdots + k_{i} = n} \prod_{j=1}^{i} p_{k_{j}}\right) -i \left(\sum_{1\leq k_1, ..., k_{i-1} \leq A} \hspace{-0.4cm} p_{n- \sum_{j}
k_{j}} \prod_{j=1 }^{i-1} p_{k_{j}}\right) \right | &\leq&   C
i^{7/2}n^{-3/2}A^{-1/2}. \label{teknik}
  \end{eqnarray}
We apply this to $A=A(n)=n^{8\varepsilon}$ in such a way
that the right-hand side of the above inequality is negligible
in comparison with $i p_{n}$. Then note that for $1\leq k_1, ..., k_{i-1}
\leq A$, we have
$$n-n^{9\varepsilon}\leq  n- \sum_{j=1}^{i-1} k_{j} \leq n.$$ Moreover, since $9\varepsilon<1$, Lemma
\ref{estimee}  gives   $p_j/p_n \rightarrow 1$
uniformly in $n-n^{9\varepsilon} \leq j \leq n$ as $n \rightarrow
\infty$. Thus, using  Lemma \ref{estimee} again
  $$i \left( 
 \sum_{1\leq k_1, ..., k_{i-1} \leq A} p_{n- \sum_{j}
k_{j}} \prod_{j=1 }^{i-1} p_{k_{j}}  \right) \sim i p_{n} \left(
\sum_{k=1}^{n^{8 \varepsilon}} p_{k}\right)^{i-1} \sim  \quad i
p_{n}\left(1-\frac{n^{-4
\varepsilon}}{\sqrt{\pi\sqrt{2}}}\right)^{i} \sim i p_{n},$$
which completes the proof of the lemma. 
\endproof

\subsection{Applications}

The ``local convergence'' given in Theorem \ref {thm:cvlocale} allows us to study ``local''
properties of large uniform dissections by reading them directly on
the critical Galton-Watson tree conditioned to survive. We will
focus our attention on the following two local properties of random
uniform dissections: Vertex degrees and face degrees. \bigskip

 Let us introduce some notation. Recall that $\mathcal{D}_n$
 stands for a uniformly distributed dissection of $P_{n+1}$. Denote by $\delta^{(n)}$
the degree of the face adjacent to the side $[1,e^{2 \textrm{i} \pi
/{(n+1)}}]$ in the random dissection $\mathcal{D}_n$ and by
$D^{(n)}$ the maximal degree of a face of $ \mathcal{D}_n$.
Similarly, denote by $\partial^{(n)}$ the number of diagonals
adjacent to the vertex corresponding to the complex number $1$ in $
\mathcal{D}_n$ and by $\Delta^{(n)}$ the maximal number of diagonals
adjacent to some vertex of $P_{n+1}$. Finally, for $b>0$, we write $ \log_b( \cdot )$ for $ \ln( \cdot )/ \ln(b)$.

We shall establish that $\delta^{(n)}$ and $\partial^{(n)}$ converge
in distribution, and read their limiting distributions on the random infinite tree $ \mathcal{T}_{\infty}$. We also provide
sharp concentration estimates on $D^{(n)}$ and $\Delta^{(n)}$,
confirming in particular a conjecture of \cite{BPS08}
concerning $ \Delta^{(n)}$.

 \subsubsection{Face degrees}


\begin{prop}
As $n$ goes to infinity, $\delta^{(n)}$
converges in distribution to the random variable $X$ with
distribution
$$\P(X=k)=(k-1) \mu_{k-1}=(k-1) \left( \frac{2-\sqrt{2}}{2}\right)^{k-2}, \qquad k \geq
3.$$\end{prop}

\begin{proof}This is an immediate consequence of Theorem \ref{thm:cvlocale}
and the construction of the critical Galton-Watson tree conditioned
to survive, after taking into account the fact that $\delta^{(n)}-1$ is the
number of children of $\varnothing$ in the dual tree of
$\mathcal{D}_n$.\end{proof}



\begin{prop}\label{thm:facedegrees}
Set $\beta=1/\mu_2=2+\sqrt{2}$. For every $c>0$, we
have
 \begin{eqnarray*}\P \left(\log_\beta(n)-c \log_\beta \log_\beta(n)\leq D^{(n)} \leq \log_\beta(n)+c \log_\beta
\log_\beta(n) \right)
& \xrightarrow[n\to\infty]{}& 0. \end{eqnarray*}
\end{prop}

\begin{proof}
By construction of the dual tree $ \mathcal{T}_{n}$ of $
\mathcal{D}_{n}$, we have $D^{(n)}-1=  \max_{u \in \t_n} k_u$. Thus, by Proposition \ref{prop:unif}, for every
measurable function $F: \R_+ \rightarrow \R_+$ we have
$$\Es{F(D^{(n)}-1)}=\mathbb{E}_\mu\left[F\left(\max_{u \in \tau} k_u
\right) \, | \, \lt=n\right].$$ The result now follows from \cite[Remark 7.3]{K:leaves}.
\end{proof}

\subsubsection{Vertex degrees}

We are now interested in another graph-theoretical property of large
uniform dissections, namely vertex degrees. Since these vertex
degrees are read on the dual tree in a more complicated fashion than
face degrees, the arguments are slightly more involved.

Recall that $\partial^{(n)}$ stands for the number of diagonals
adjacent to the vertex corresponding to the complex number $1$ in
the uniform dissection $ \mathcal{D}_n$.

\begin{prop}\label{prop:somme2geom}
As $n$ goes to infinity, $\partial^{(n)}$ converges
in distribution to the sum of two independent geometric random variables of
parameter $1-\mu_0=\sqrt{2}-1$, $i.e.$ for any $k \geq 0$ we have
 \begin{eqnarray*}  \P(\partial ^{(n)} = k) & \xrightarrow[n\to\infty]{} & (k+1)\mu_0^2(1-\mu_0)^{k}.  \end{eqnarray*}
 \end{prop}
 See also \cite{BPS08} for a closely related result.

\proof If $\tau$ is a plane tree, we introduce  the length $\ell(\tau)$ of the left-most path in $\tau$ starting at $\varnothing$ (that is following left-most children
until we reach a leaf),
$$ \ell(\tau) = \max \{ i \geq 0 : 1_{i} \in \tau \}, \quad \mbox{ where }1_{i} = {1...1} \quad (i \mbox{ times}) \textrm{ with } 1_0=\varnothing.$$
Using the bijection between a dissection and its dual tree, it is
easy to see that the number of diagonals adjacent to the vertex $1$
of the random dissection $ \mathcal{D}_{n}$ is
 $\ell( \mathcal{T}_{n})-1$. By Theorem
\ref{thm:cvlocale}, for every $k \geq 1$, $\left[\t_{n}\right]_{k}
\to \left[  \mathcal{T}_{\infty} \right]_{k}$ as $n\to\infty$. It
follows that $\partial^{(n)} = \ell( \mathcal{T}_{n})-1$ converges
in distribution towards $\ell( \mathcal{T}_{\infty})-1$. Let us
identify the distribution of this variable using the first
description of the law $ \widehat{ \mathbb{P}}_{\mu}$: The length
$\ell( \mathcal{T}_{\infty})-1$ can be decomposed into $$\ell(
\mathcal{T}_{\infty})-1 = (\ell_{1}-1) + \ell_{2},$$ where
$\ell_{1}$ is the smallest integer $i \geq 1$ such that the element
$1_{i} = 1...1$ ($i$ times) is not on the \emph{spine} of $
\mathcal{T}_{\infty}$ and $\ell_{2}= \ell(
\mathcal{T}_{\infty})-(\ell_{1}-1)$ is the length of the left-most
path in the critical $\mu$-Galton-Watson tree grafted on the left of
$1_{i}$. By the description in Section \ref{cgwi}, the two
variables $\ell_{1}-1$ and $\ell_{2}$ are independent. It is straightforward that $\ell_2$ is distributed according to a geometric variable of parameter $\sqrt{2}-1$. Let us now turn to $\ell_1 -1$. Recall the notation introduced in Section \ref{cgwi}. For $k \geq 0$, we have: 
 \begin{eqnarray*} \mathbb{P}(\ell_1  \geq k+1) &= &  \mathbb{P}(U_1=1, U_2=1, \ldots , U_k=1)\\
 &=& \prod_{i=0}^{k-1} \left(\sum_{j=1}^\infty \mathbb{P}(D_i = j) \mathbb{P}(U_{i+1} = 1 \mid D_{i}=j) \right)\\
 &=& \left(\sum_{j=2}^{\infty} (1-2^{-1/2})^{j-1}\right)^{k} = (1-\mu_0)^k. \end{eqnarray*}
 We thus see that $\ell_1-1$ is also geometric with parameter $ \sqrt {2}-1$ and the desired result follows.
\endproof

Recall that $\Delta^{(n)}$ stands for the maximal number of
diagonals adjacent to a vertex of $P_{n+1}$. The following
theorem proves a conjecture of \cite{BPS08}.

\begin{thm}\label{thm:vertexdegrees}Set $b=1/(1-\mu_0)=\sqrt{2}+1$. For every $c>0$, we have
 \begin{eqnarray*}\P \left(\Delta^{(n)} \geq \log_b(n)+(1+c) \log_b \log_b(n) \right)
& \xrightarrow[n\to\infty]{}& 0. \end{eqnarray*}
\end{thm}

\proof Let $p,n \geq 1$ be integers. By rotational invariance, the degrees of  the vertices of $ \mathcal{D}_n$ are identically
distributed random variables. It follows that
 \begin{eqnarray*}\Pr{\Delta^{(n)} \geq p}
 &\leq& (n+1)\P\left(
\partial^{(n)} \geq p\right).  \end{eqnarray*}
We have already noticed that the number of
diagonals adjacent to the vertex $1$ of the random dissection $
\mathcal{D}_{n}$ corresponds to $\ell( \mathcal{T}_{n})-1$, where $\ell( \t_n)$ denotes the length of the left-most path in
$ \mathcal{T}_{n}$ starting at $\varnothing$. 
Thus, by Proposition \ref{prop:unif},
 \begin{eqnarray*}(n+1)\P\left(
\partial^{(n)}_0 \geq p\right)&=& (n+1) \mathbb{P}_{\mu} \left( \ell ( \tau) \geq p+1 \mid \lambda(\tau)=n \right) \label{tobeevaluate}
 \end{eqnarray*}
We now estimate the right-hand side and show
that it tends to $0$ when $n \rightarrow \infty$ and $p=p_{n} =\log_b(n)+(1+c) \log_b
\log_b(n)$ for $c>0$. If $\ell(\tau)=p$, define $\theta(\tau) =
k_{\varnothing} + k_{1} + k_{1_{2}} + ... + k_{1_{p-1}}- p$ ($\theta (\tau)$ can be interpreted as the total number of those children of vertices in the left-most path that are not in that path). Note that under $\Pmu$, $\ell(\tau)$ is distributed according
to a geometric random variable of parameter $\sqrt{2}-1$. In
particular, for $\a=|4/\log(1-\mu_0)|$,
\begin{equation}\label{eq:0}n^3\Prmu{\ell(\tau) \geq \a\log(n)} \quad \mathop{\longrightarrow}_{n \rightarrow \infty}\quad 0.\end{equation}
Note also that for positive integers $j,k$ we have
$\Prmu{\theta(\tau)=j  \mid \ell(\tau)=k}= \P(Y^{*k}=j)$, where
$Y^{*k}$ is distributed as the sum of $n$ independent random
variables distributed according to $ \gamma(j)= \mu_{j+1}/ \mu([1, \infty])$.

Choose $\varepsilon>0$ such that the conclusion of Lemma \ref{ui} holds. We first claim that if $A_{n} = \{ \theta(\tau) \leq n^\varepsilon\}$, then $
\mathbb{P}_{\mu} \left( A_{n} \right) \geq 1 - n^{-3}$ for $n$ large
enough. To this end, write
\begin{eqnarray*} n^3 \mathbb{P}_{\mu} \left( \theta(\tau) > n^\varepsilon \right)
 &\leq&  n^3 \Prmu{\ell(\tau)> \a \log(n)} + n^3 \sum_{j=1}^{\lfloor
\a\log(n)\rfloor} \Prmu{\theta(\tau) > n^\varepsilon \, \| \,
\ell(\tau)=j} \Prmu{\ell(\tau)=j}\\
&\leq& n^3 \Prmu{\ell(\tau)> \a \log(n)}  +  n^3 \a \log(n)
\Prmu{\theta(\tau)
> n^\varepsilon \, \| \, \ell(\tau)=\lfloor \a \log(n)\rfloor}.
\end{eqnarray*}
The first term tends to zero by \eqref{eq:0}, and the second one as
well by the previous description of the law of $\theta(\tau)$ under
the conditional probability distribution $\Prmu{ \, \cdot  \mid
\ell(\tau)=k}$ and a standard large deviation inequality. Thus our claim holds at it follows that, for $n$ large enough, \begin{eqnarray} \label{zob1} (n+1) \mathbb{P}_{\mu}(A_{n}^c \mid \lambda(\tau)=n)  \leq
(n+1)n^{-3}/ \mathbb{P}_{\mu}(\lambda(\tau)=n) & \xrightarrow[n\to\infty]{}& 0  \end{eqnarray} 
by Lemma \ref {estimee}. We now consider the event $A_{n}$. We have
\begin{eqnarray*}  && \Prmuln{  \{\ell(\tau) = p \} \cap A_{n}} \\
&& \qquad \qquad \qquad \leq \qquad  \mu_0 \sum_{\begin{subarray}{c} r_0, ..., r_{p-1} \geq 0 \\
\sum r_{j} \leq n^\varepsilon \end{subarray}} \mu_{r_0+1} \cdots
\mu_{r_{p-1}+1} \frac{\P_{\mu,r_0+r_1+\cdots+r_{p-1}}\left( \lt = n-1
\right)}{\Pmu(\lt=n)}.
 \end{eqnarray*}
 We can then apply Lemma \ref{ui} to get that the quotient in the last display is bounded above by
 some constant $C_{2}$ times $\theta(\tau)=r_{0} + ... + r_{p}$, so that
  \begin{eqnarray*}
\Prmuln{\ell(\tau) = p, \theta(\tau) \leq n^\varepsilon }&\leq &
C_{2} \mu_0 \sum_{\begin{subarray}{c} r_0, \ldots, r_{p-1} \geq 0
\end{subarray}} \mu_{r_0+1} \cdots \mu_{r_{p-1}+1}
(r_{0} + \ldots + r_{p-1}) \\
&=& C_{2} p \mu_0^2(1-\mu_{0})^ {p-1},
\end{eqnarray*}
where we have successively calculated these sums by using $$\sum_{k=1}^\infty \mu_{k+1}=1-\mu_0,\qquad\sum_{k=1}^\infty k \mu_{k+1}=\mu_0.$$
Consequently, setting $p_n=\log_b(n)+(1+c) \log_b \log_b(n)$ we
deduce from the above estimates that
 \begin{eqnarray*} (n+1) \mathbb{P}_{\mu} \left( \ell ( \tau) \geq p_n,  \theta(\tau) \leq n^\varepsilon \mid \lambda(\tau)=n \right) & \xrightarrow[n\to\infty]{}& 0,  \end{eqnarray*} which together with \eqref{zob1} completes the proof
of the theorem.
\endproof

\begin{rem}It is possible to simplify the proof of the preceding theorem by
using the fact that there exists $C>0$ such that for every integers
$i,n \geq 1$ $$\P_{\mu,i}(\lambda( \mathfrak{f})=n) \leq  C i \times
\Pmu(\lt=n).$$ The proof of this fact easily follows from results
contained in \cite{K:leaves}. However, we have preferred to prove a
weaker form of this inequality by elementary arguments without
relying on more involved results.
\end{rem}

 In \cite{BPS08}, it is shown that for every $c>0$
 \begin{eqnarray*}\P \left(\Delta^{(n)} \leq \log_b(n)-(2+c) \log_b \log_b(n) \right)
&\xrightarrow[n\to\infty]{} & 0. \end{eqnarray*} Using the
connection between uniform dissections and Galton-Watson
trees conditioned on their number of leaves, it is possible to
refine the above lower bound and to replace $(2+c)$ by $c$. However,
we believe that the optimal concentration result is given by the
following conjecture:

\begin{conj} For every $c>0$ we have
 \begin{eqnarray*}\P\left( \left|\Delta^{(n)}- (\log_b(n)+ \log_b \log_b(n))\right| > c \log_b\log_b(n)\right)
&\xrightarrow[n\to\infty]{} & 0.\end{eqnarray*}
\end{conj}

If the degrees of vertices in $ \mathcal{D}_n$ were independent, this
concentration result would hold. However, the difficulty comes from
the fact that this independence property does not exactly hold. Let us mention that
this conjecture (for a different value of $b$) has been proved in
the particular case of triangulations using generating function
methods \cite{GW00}.

\bigskip It is worth pointing out that although $\delta^{n}$ and
$\partial^{(n)}$ have a similar limiting distribution (roughly
speaking a size-biased geometric distribution), the maximal degree
of a face and the maximal degree of a vertex in a random uniform
dissection possess a different concentration
behavior: $D^{(n)}$ is strongly concentrated around $\log_\b(n) +
o(\log\log(n))$, whereas $\Delta^{(n)}$ is (conjecturally) strongly concentrated
around $\log_b(n) + \log_b\log_b(n) + o(\log\log(n))$. This comes from the heuristic observation that a ``typical'' vertex has a limiting distribution which is given by a sized-biased geometric distribution, whereas   a ``typical'' face of $ \mathcal{D}_n$  has a limiting distribution which is a geometric distribution (which is not size-biased). Let us give some details.

 We start with the vertex degree. By Proposition \ref{prop:somme2geom} and by rotational invariance, a ``typical'' vertex has a limiting distribution which is given by a sized-biased geometric distribution. This is why $\Delta^{(n)}$ should have the same concentration behavior as $n$ independent random variables distributed as size-biased geometric laws, that is, $\Delta^{(n)}$ should be concentrated around $\log_{b}(n) + \log_{b}\log_{b}(n)+o ( \log \log(n))$.

The situation is however different is the case of face degrees. Choosing the face adjacent to $[1,e^{2 \mathrm{i}\pi /(n+1)}]$ introduces a size-biasing in the distribution of a "typical" face (indeed, the face containing a given side of $P_n$ is not a typical face but a size-biased one; a typical face would be a face chosen "uniformly" among all faces of the dissection). In other words,  a
typical face of $ \mathcal{D}_n$ follows a geometric distribution, but
$\delta^{(n)}$ is a size-biased distribution of a
typical face of $ \mathcal{D}_n$.  This is why
$D^{(n)}$ follows the same concentration behavior as $n$ independent
geometric variables (which are not size-biaised), and hence explains why $D^{(n)}$ is concentrated around $\log_{\beta}(n)+ o( \log\log n)$.

\begin {tabular}{l p{3mm} l}
 D\'epartement de Math\'ematiques et Applications,  && Laboratoire de mathématiques, \\
 \'Ecole Normale Sup\'erieure. &&UMR 8628 CNRS.\\
45 rue d'Ulm && Université Paris-Sud\\
 
 75230 Paris cedex 05, France && 91405 ORSAY Cedex, France
\end {tabular}

\medbreak 
\noindent nicolas.curien@ens.fr\\
\noindent igor.kortchemski@normalesup.org

\end{document}